\documentclass{amsart}

\usepackage{amscd}
\usepackage{amsmath}
\usepackage{amssymb}
\usepackage{amsthm}
\usepackage{epsf}
\usepackage{latexsym}
\usepackage{verbatim}
\usepackage[all, cmtip]{xy}
\usepackage{tikz}
\usetikzlibrary{positioning}
\usetikzlibrary{matrix}

\tikzstyle{bsq}=[rectangle, draw, thick, minimum width=1cm, minimum height=1cm]
\tikzstyle{bver}=[rectangle, draw, thick, minimum width=1cm, minimum height=2cm]
\tikzstyle{bhor}=[rectangle, draw, thick, minimum width=2cm, minimum height=1cm]

\newtheorem{theorem}{Theorem}[section]
\newtheorem{definition}[theorem]{Definition}
\newtheorem{lemma}[theorem]{Lemma}
\newtheorem{conjecture}[theorem]{Conjecture}

\newtheorem{proposition}[theorem]{Proposition}

\newtheorem{varexample}[theorem]{Example}

\newtheorem*{MRC}{Maximal Rank Conjecture}

\newtheorem*{ShapeLemma}{Shape Lemma for Minima}
\theoremstyle{definition}
\newtheorem{remark}[theorem]{Remark}
\newtheorem{notation}[theorem]{Notation}

\newcommand{\Spec}{\mathrm{Spec}\,}

\newcommand{\PP}{\mathbb{P}}

\newcommand{\ZZ}{\mathbb{Z}}

\newcommand{\rr}{\left \lceil \frac{r}{2} \right \rceil}

\newcommand{\cA}{\mathcal{A}}

\newcommand{\cG}{\mathcal{G}}
\newcommand{\cL}{\mathcal{L}}

\newcommand{\ord}{\operatorname{ord}}

\newcommand{\trop}{\operatorname{trop}}
\newcommand{\ddiv}{\operatorname{div}}

\newcommand{\val}{\operatorname{val}}

\newcommand{\Sym}{\operatorname{Sym}}
\newcommand{\outdeg}{\mathrm{outdeg}}

\newcommand{\an}{\mathrm{an}}

\newenvironment{example}{\begin{varexample}
\begin{normalfont}}{\end{normalfont}
\end{varexample}}
\begin{document}
\title{Tropical Independence II:  The Maximal Rank Conjecture for Quadrics}
\author{David Jensen}
\author{Sam Payne}
\date{}
\bibliographystyle{alpha}

\maketitle

\begin{abstract}
Building on our earlier results on tropical independence and shapes of divisors in tropical linear series, we give a tropical proof of the maximal rank conjecture for quadrics.  We also prove a tropical analogue of Max Noether's theorem on quadrics containing a canonically embedded curve, and state a combinatorial conjecture about tropical independence on chains of loops that implies the maximal rank conjecture for algebraic curves.
\end{abstract}

\section{Introduction}

Let $X \subset \PP^r$ be a smooth curve of genus $g$, and recall that a linear map between finite dimensional vector spaces has \emph{maximal rank} if it is either injective or surjective.  The kernel of the restriction map
\[
\rho_m: H^0(\PP^r, \mathcal{O}(m)) \rightarrow H^0(X, \mathcal{O}(m)|_X)
\]
is the space of homogeneous polynomials of degree $m$ that vanish on $X$.  The conjecture that $\rho_m$ should have maximal rank for sufficiently general embeddings of sufficiently general curves, attributed to Noether in \cite[p.\,4]{ArbarelloCiliberto83},\footnote{Noether considered the case of space curves in \cite[\S8]{Noether82}.  See also \cite[pp.\,172--173]{CES25} for hints toward Noether's understanding of the general problem.} was studied classically by Severi \cite[\S10]{Severi15}, and popularized by Harris \cite[p.\,79]{Harris82}.

\begin{MRC}
Let $V \subset \cL(D_X)$ be a general linear series of rank $r \geq 3$ and degree $d$ on a general curve $X$ of genus $g$.  Then the multiplication maps
\[
\mu_m: \Sym^m V \rightarrow \cL(mD_X)
\]
have maximal rank for all $m$.
\end{MRC}

\noindent Our main result gives a combinatorial condition on the skeleton of a curve over a nonarchimedean field to ensure the existence of a linear series such that $\mu_2$ has maximal rank.  In particular, whenever the general curve of genus $g$ admits a nondegenerate embedding of degree $d$ in $\PP^r$ then the image of a general embedding is contained in the expected number of independent quadrics.
\medskip

Let $\Gamma$ be a chain of loops connected by bridges with admissible edge lengths, as defined in \S\ref{Section:TheGraph}.  See Figure~\ref{Fig:TheGraph} for a schematic illustration, and note that our conditions on the edge lengths are more restrictive than those in \cite{tropicalBN, tropicalGP}.

\begin{theorem}
\label{thm:mainthm}
Let $X$ be a smooth projective curve of genus $g$ over a nonarchimedean field such that the minimal skeleton of the Berkovich analytic space $X^{\an}$ is isometric to $\Gamma$.  Suppose $r \geq 3$, $\rho(g,r,d) \geq 0$, and $d < g + r$.  Then there is a very ample complete linear series $\cL(D_X)$ of degree $d$ and rank $r$ on $X$ such that the multiplication map
\[
\mu_2: \Sym^2 \cL(D_X) \rightarrow \cL(2D_X)
\]
has maximal rank.
\end{theorem}

\noindent Such curves do exist, over fields of arbitrary characteristic, and the condition that $X^\an$ has skeleton $\Gamma$ ensures that $X$ is Brill--Noether--Petri general \cite{tropicalGP}.  As explained in \S\ref{sec:prelim}, to prove the maximal rank conjecture for fixed $g$, $r$, $d$, and $m$ over an algebraically closed field of given characteristic, it is enough to produce a single linear series $V \subset \cL(D_X)$ on a single Brill--Noether--Petri general curve over a field of the same characteristic for which $\mu_m$ has maximal rank.  In particular, the maximal rank conjecture for $m = 2$, and arbitrary $g$, $r$, and $d$, follows from Theorem~\ref{thm:mainthm}.

Surjectivity of $\mu_m$ for small values of $m$ can often be used to prove surjectivity for larger values of $m$.  See, for instance, \cite[pp. 140--141]{ACGH}.  In characteristic zero, when Theorem~\ref{thm:mainthm} gives surjectivity of $\mu_2$, we apply standard arguments from linear series on curves to deduce surjectivity of $\mu_m$ for all $m$.

\begin{theorem}
\label{thm:allm}
Let $X$ and $D_X$ be as in Theorem~\ref{thm:mainthm}, and suppose $\mu_2$ is surjective. Then $\mu_m$ is surjective for all $m \geq 2$.
\end{theorem}

\noindent  This confirms the maximal rank conjecture for all $m$ in the range where $\mu_2$ is surjective.  In particular, a general nondegenerate embedding of a general curve is projectively normal if and only if ${r + 2 \choose 2} \geq 2d -g +1$.

\begin{remark}
The maximal rank conjecture is known, for all $m$, when $r=3$ \cite{BallicoEllia87b}, and in the non-special case $d \geq r + g$ \cite{BallicoEllia87}.  There is a rich history of partial results on the maximal rank conjecture for $m = 2$, including some with significant applications.  Voisin proved the case of adjoint bundles of gonality pencils and deduced the surjectivity of the Wahl map for generic curves \cite[\S4]{Voisin92}.  Teixidor proved that $\mu_2$ is injective for all linear series on the general curve when $d < g + 2$ \cite{Teixidor03}, over fields of characteristic not equal to two. Farkas proved the case where $\rho(g,r,d)$ is zero and $\dim \Sym^2 \cL(D_X) = \dim \cL(2D_X)$, and used this to deduce an inifinite sequence of counterexamples to the slope conjecture \cite[Theorem~1.5]{Farkas09}.    Another special case is Noether's theorem on canonically embedded curves, discussed below.  Furthermore, Larson has proved an analogue of the maximal rank conjecture for hyperplane sections of curves \cite{Larson12}.  This is only a small sampling of prior work on the maximal rank conjecture.  We note, in particular, that the literature contains many short articles by Ballico based on classical degeneration methods, and the maximal rank conjecture for quadrics in characteristic zero appears as Theorem~1 in \cite{Ballico12}.
\end{remark}

Two key tools in the proof of Theorem~\ref{thm:mainthm} are the lifting theorem from \cite{LiftingDivisors} and the notion of \emph{tropical independence} developed in \cite{tropicalGP}.  The lifting theorem allows us to realize any divisor $D$ of rank $r$ on $\Gamma$ as the tropicalization of a divisor $D_X$ of rank $r$ on $X$.  Our understanding of tropical linear series on $\Gamma$, together with the nonarchimedean Poincar\'{e}-Lelong formula, produces rational functions $\{f_0, \ldots, f_r \}$ in the linear series $\cL(D_X)$ whose tropicalizations $\{ \psi_0, \ldots, \psi_r \}$ are a specific well-understood collection of piecewise linear functions on $\Gamma$.  We then show that a large subset of the piecewise linear functions $\{ \psi_i + \psi_j \}_{0 \leq i \leq j \leq r}$ is tropically independent.  Since $\psi_i + \psi_j$ is the tropicalization of $f_i \cdot f_j$, the size of this subset is a lower bound for the rank of $\mu_2$, and this is the bound we use to prove Theorem~\ref{thm:mainthm}.

There is no obvious obstruction to proving the maximal rank conjecture in full generality using this approach, although the combinatorics become more challenging as the parameters increase.  We state a precise combinatorial conjecture in \S\ref{Section:TheGraph}, which, for any given $g$, $r$, $d$, and $m$, implies the maximal rank conjecture for the same $g$, $r$, $d$, and $m$.  We prove this conjecture not only for $m = 2$, but also for $md < 2g + 4$.  (See Theorem~\ref{Thm:LowDegree}.)
\medskip

We also present advances in understanding multiplication maps by tropical methods on skeletons other than a chain of loops.  Recall that Noether's theorem on canonically embedded curves says that $\mu_2 : \Sym^2 \cL(K_X) \rightarrow \cL(2K_X)$ is surjective whenever $X$ is not hyperelliptic.  This may be viewed as a strong form of the maximal rank conjecture for quadrics in the case where $r = g-1$ and $d = 2g-2$.

On the purely tropical side, we prove an analogue of Noether's theorem for trivalent, 3-edge-connected graphs.

\begin{theorem}
\label{Thm:TropicalMaxNoether}
Let $\Gamma$ be a trivalent, 3-edge-connected metric graph. Then there is a tropically independent set of $3g- 3$ functions in $2R(K_\Gamma)$.
\end{theorem}

\noindent Furthermore, we prove the appropriate lifting statements to leverage this tropical result into a maximal rank statement for canonical embeddings of curves with trivalent and 3-connected skeletons.

\begin{theorem}\label{Thm:MaxNoether}
Let $X$ be a smooth projective curve of genus $g$ over a nonarchimedean field such that the minimal skeleton $\Gamma$ of $X^{\an}$ is trivalent and 3-edge-connected with first Betti number $g$.  Then there are $3g-3$ rational functions in the image of $\mu_2 : Sym^2 \cL (K_X) \to \cL (2K_X)$ whose tropicalizations are tropically independent.  In particular,  $\mu_2$ is surjective.
\end{theorem}

\noindent The last statement, on surjectivity of $\mu_2$, also follows from Noether's theorem, because trivalent, 3-edge connected graphs are never hyperelliptic \cite[Lemma~5.3]{BakerNorine09}.

\begin{remark}
The present article is a sequel to \cite{tropicalGP}, further developing the method of tropical independence.  This is just one aspect of the tropical approach to linear series, an array of techniques for handling degenerations of linear series over a one parameter family of curves where the special fiber is not of compact type, combining discrete methods with computations on skeletons of Berkovich analytifications.  Seminal works in the development of this theory include \cite{BakerNorine07, Baker08, AminiBaker15}.  Combined with techniques from $p$-adic integration, this method also leads to uniform bounds on rational points for curves of fixed genus with small Mordell--Weil rank \cite{KRZB15}.

This tropical approach is in some ways analogous to the theory of limit linear series, developed by Eisenbud and Harris in the 1980s, which systematically studies the degeneration of linear series to singular curves of compact type \cite{EisenbudHarris86}.  This theory led to simplified proofs of the Brill--Noether and Gieseker--Petri theorems \cite{EisenbudHarris83c}, along with many new results about the geometry of curves, linear series, and moduli \cite{EisenbudHarris87, EisenbudHarris87b, EisenbudHarris87c, EisenbudHarris89}.  Tropical methods have also led to new proofs of the Brill--Noether and Gieseker--Petri theorems \cite{tropicalBN, tropicalGP}.  Some progress has been made toward building frameworks that include both classical limit linear series and also generalizations of limit linear series for curves not of compact type \cite{AminiBaker15, Osserman14b, Osserman14}, which are helpful for explaining connections between the tropical and limit linear series proofs of the Brill--Noether theorem.  These relations are also addressed in \cite[Remark~1.4]{tropicalGP} and \cite{CLMTiB14}.  The nature of the relations between the tropical approach and more classical approaches for results involving multiplication maps, such as the Gieseker--Petri theorem and other maximal rank results, remain unclear, as do the relations between such basic and essential facts as the Riemann--Roch theorems for algebraic and tropical curves.

Note that several families of curves appearing in proofs of the Brill--Noether and Gieseker--Petri theorems are not contained in the open subset of $\overline{\mathcal{M}}_g$ for which the maximal rank condition holds.  For example, the sections of $K3$ surfaces used by Lazarsfeld in his proof of the Brill--Noether and Gieseker--Petri theorems without degenerations \cite{Lazarsfeld86} do not satisfy the maximal rank conjecture for $m = 2$ \cite[Theorem~0.3 and Proposition~3.2]{Voisin92}.  Furthermore, the stabilizations of the flag curves used by Eisenbud and Harris are limits of such curves \cite[Proposition~7.2]{FarkasPopa05}.
\end{remark}

\noindent \textbf{Acknowledgments.} We thank D.~Abramovich, E.~Ballico, and D.~Ranganathan for helpful comments on an earlier version of this paper.

The second author was supported in part by NSF CAREER DMS--1149054 and is grateful for ideal working conditions at the Institute for Advanced Study in Spring 2015.

\section{Preliminaries} \label{sec:prelim}

Recall that a general curve $X$ of genus $g$ has a linear series of rank $r$ and degree $d$ if and only if the Brill--Noether number
\[
\rho(g,r,d) = g - (r+1)(g-d+r)
\]
is nonnegative, and the scheme $\cG^r_d(X)$ parametrizing its linear series of degree $d$ and rank $r$ is smooth of pure dimension $\rho(g,r,d)$.  This scheme is irreducible when $\rho(g,r,d)$ is positive, and monodromy acts transitively when $\rho(g,r,d) = 0$.  Therefore, if $U \subset \mathcal{M}_g$ is the dense open set parametrizing such \emph{Brill--Noether--Petri general} curves, then $\cG^r_d(U)$, the universal linear series of rank $r$ and degree $d$ over $U$, is smooth and irreducible of relative dimension $\rho(g,r,d)$.  The \emph{general linear series of degree $d$ and rank $r$ on a general curve of genus $g$} appearing in the statement of the maximal rank conjecture refers simply to a general point in the irreducible space $\cG^r_d(U)$.

When $X$ is Brill--Noether--Petri general and $D_X$ is a basepoint free divisor of rank at least 1, the basepoint free pencil trick shows that its multiples $mD_X$ are nonspecial for $m \geq 2$ (see Remark~\ref{Rem:Nonspecial}).  Therefore, by standard upper semincontinuity arguments from algebraic geometry and the fact that $\cG^r_d$ is defined over $\Spec \ZZ$, to prove the maximal rank conjecture for fixed $g$, $r$, $d$, and $m$, over an arbitrary algebraically closed field of given characteristic, it suffices to produce a single Brill--Noether--Petri general curve $X$ of genus $g$ over a field of the same characteristic with a linear series $V \subset \cL(D_X)$ of degree $d$ and rank $r$ such that $\mu_m$ has maximal rank.  As mentioned in the introduction, the maximal rank conjecture is known when the linear series is nonspecial.  In the remaining cases, the general linear series is complete, so we can and do assume that $V = \cL(D_X)$.

\begin{remark}\label{Rem:Nonspecial}
Suppose $D_X$ is a basepoint free special divisor of rank $r \geq 1$ on a Brill--Noether--Petri general curve $X$.  The fact that $mD_X$ is nonspecial for $m \geq 2$ is an application of the basepoint free pencil trick, as follows.  Choose a basepoint free pencil $V \subset \cL(D_X)$.  Then the trick identifies $\cL(K_X - 2D_X)$ with the kernel of the multiplication map
\[
\mu: V \otimes \cL(K_X - D_X) \rightarrow \cL(K_X).
\]
The Petri condition says that this multiplication map is injective, even after replacing $V$ by $\cL(D_X)$.  Therefore, there are no sections of $K_X$ that vanish on $2D_X$ and hence no sections that vanish on $mD_X$ for $m \geq 2$, which means that $mD_X$ is nonspecial.
\end{remark}

\begin{remark}
When $r \geq 3$ and $\rho(g,r,d) \geq 0$, the general linear series of degree $d$ on a general curve of genus $g$ defines an embedding in $\mathbb{P}^r$, and hence the conjecture can be rephrased in terms of a general point of the corresponding component of the appropriate Hilbert scheme.  One can also consider analogues of the maximal rank conjecture for curves that are general in a given irreducible component of a given Hilbert scheme, rather than general in moduli.  However, the maximal rank condition can fail when the Hilbert scheme in question does not dominate $\mathcal M_g$.  Suppose, for example, that $X$ is a curve of genus 8 and degree 8 in $\mathbb{P}^3$.  Then $h^0 (\mathcal{O}_X (2)) = 9$, and hence $X$ is contained in a quadric surface.  It follows that the kernel of $\mu_3$ has dimension at least 4, and therefore $\mu_3$ is not surjective.  This does not contradict the maximal rank conjecture, since the general curve of genus $8$ has no linear series of rank $3$ and degree $8$.
\end{remark}

\begin{proof}[Proof of Theorem~\ref{thm:allm}]
We assume $\mu_2$ is surjective.  Suppose $r \geq 4$.  We begin by showing that $\mu_3$ is surjective.  Note the polynomial identity
$$ {{r+2}\choose{2}} - (2d-g+1) = {{d-g}\choose{2}} - {{g-d+r}\choose{2}} - \rho (g,r,d) . $$
(This identity reappears as Lemma \ref{lem:identity}, in the special case $\rho(g,r,d) = 0$.)  By assumption, the left hand side is nonnegative, as are $\rho (g,r,d)$ and $g-d+r$.  It follows that $d \geq g$.  By \cite[Exercise B-6, p. 138]{ACGH}\footnote{The statement of the exercise is missing a necessary hypothesis, that $\mathcal{D}$ has rank at least 3.  The solution following the hint requires the uniform position lemma, which is known for $r \geq 3$ in characteristic zero \cite{Harris80} and, over arbitrary fields, when $r \geq 4$ \cite{Rathmann87}.}, it follows that the dimension of the linear series spanned by sums of divisors in $\vert D_X \vert$ and $\vert 2D_X \vert$ is at least
$$ \min \{ 4d-2g, 3d-g \} = 3d-g. $$  Therefore, if $\mu_2$ is surjective then $\mu_3$ is also surjective.  

We now show, by induction on $m$, that $\mu_m$ is surjective for all $m>3$.  Let $V \subset \cL (D_X)$ be a basepoint free pencil.  By the basepoint free pencil trick, we have an exact sequence
$$ 0 \to \cL ((m-1)D_X) \to V \otimes \cL (mD_X) \to \cL ((m+1)D_X) . $$
Since $(m-1)D_X$ and $mD_X$ are both nonspecial, the image of the right hand map has dimension
$$ 2(md-g+1) - ((m-1)d-g+1) = (m+1)d-g+1, $$
hence it is surjective.

It remains to consider the cases where $r = 3$.  By assumption, the divisor $D_X$ is special, so $d < g + 3$.  Furthermore, $\mu_2$ is surjective, so $2d - g + 1 \leq 10$, and $\rho(g,r,d) \geq 0$, so $3g \leq 4d - 12$.  This leaves exactly two possibilities for $(g,d)$, namely $(4,6)$ and $(5,7)$.  In each of these cases, $h^1(\mathcal{O}(D_X)) = 1$ and, since $X$ is Brill--Noether general, $\mathrm{Cliff}(X) = \left \lfloor \frac{g-1}{2} \right \rfloor.$  Then $d = 2g + 1 - h^1(\mathcal{O}(D_X)) - \mathrm{Cliff}(X)$ and hence $\mathcal{O}(D_X)$ gives a projectively normal embedding, by \cite[Theorem~1]{GreenLazarsfeld86}.
\end{proof}


Since we are trying to produce a single sufficiently general curve of each genus over a field of each characteristic, we may, for simplicity, assume that we are working over an algebraically closed field  that is spherically complete with respect to a valuation that surjects onto the real numbers.  Any metric graph $\Gamma$ of first Betti number $g$ appears as the skeleton of a smooth projective genus $g$ curve $X$ over such a field (see, for instance, \cite{acp}).

Recall that the skeleton is a subset of the set of valuations on the function field of $X$, and evaluation of these valuations, also called tropicalization, takes each rational function $f$ on $X$ to a piecewise linear function with integer slopes on $\Gamma$, denoted $\trop(f)$.

Our primary tool for using the skeleton of a curve and tropicalizations of rational functions to make statements about ranks of multiplication maps is the notion of tropical independence developed in \cite{tropicalGP}.

\begin{definition}
A set of piecewise linear functions $\{ \psi_0, \ldots, \psi_r \}$ on a metric graph $\Gamma$ is \emph{tropically dependent} if there are real numbers $b_0, \ldots, b_r$ such that for every point $v$ in $\Gamma$ the minimum
\[
\min \{\psi_0(v) + b_0, \ldots, \psi_r(v) + b_r \}
\]
occurs at least twice.  If there are no such real numbers then $\{ \psi_0, \ldots, \psi_r \}$ is \emph{tropically independent}.
\end{definition}

\noindent One key basic property of this notion is that if $\{ \trop(f_i) \}_{i}$ is tropically independent on $\Gamma$, then the corresponding set of rational functions $\{ f_i \}_{i}$ is linearly independent in the function field of $X$ \cite[\S3.1]{tropicalGP}.  Note also that if $f$ and $g$ are rational functions, then $\trop(f \cdot g) = \trop(f) + \trop(g)$.

\begin{remark}
Adding a constant to each piecewise linear function does not affect the tropical independence of a given collection.  When $\{ \psi_0, \ldots, \psi_r \}$ is tropically dependent, we often replace each $\psi_i$ with $\psi_i + b_i$ and assume that the minimum of the set $\{ \psi_0(v), \ldots, \psi_r(v) \}$ occurs at least twice at every point $v \in \Gamma$.
\end{remark}

\begin{lemma}
\label{Lem:MultiplicationMaps}
Let $D_X$ be a divisor on $X$, and let $\{f_0, \ldots, f_r\}$ be rational functions in $\cL(D_X)$.  If there exist $k$ multisets $I_1 , \ldots, I_k \subset \{0,\ldots,r\}$, each of size $m$, such that $\{\sum_{i \in I_j} \trop(f_i) \}_j$ is tropically independent, then the multiplication map
\[
\mu_m: Sym^m \cL(D_X) \rightarrow \cL(mD_X)
\]
has rank at least $k$.
\end{lemma}

\begin{proof}
The tropicalization of $\prod_{i \in I_j} f_i$ is the corresponding sum $\sum_{i \in I_j} \trop(f_i)$.  If these sums $\{\sum_{i \in I_j} \trop(f_i)\}_j$ are tropically independent then the rational functions $\{ \prod_{i \in I_j} f_i \}_j$ are linearly independent.  These $k$ rational functions are in the image of $\mu_m$, and the lemma follows.
\end{proof}

\begin{remark}
If $f_0, \ldots, f_r$ are rational functions in a linear series $\cL(D_X)$, and $b_0, \ldots, b_r$ are real numbers, then the pointwise minimum
\[
\theta = \min\{ \trop(f_0) + b_0, \ldots, \trop(f_r) + b_r \}
\]
is the tropicalization of a rational function in $\cL(D_X)$.  The rational function may be chosen of the form $a_0 f_0 + \cdots + a_r f_r$ where $a_i$ is a sufficiently general element of the ground field such that $\val(a_i) = b_i$.
\end{remark}

We will also repeatedly use the following basic fact about the shapes of divisors associated to a pointwise minimum of functions in a tropical linear series.

\medskip

\begin{ShapeLemma} \cite[Lemma~3.4]{tropicalGP}
Let $D$ be a divisor on a metric graph $\Gamma$, with $\psi_0, \ldots, \psi_r$ piecewise linear functions in $R(D)$, and let
\[
\theta = \min \{ \psi_0, \ldots, \psi_r \}.
\]
Let $\Gamma_j \subset \Gamma$ be the closed set where $\theta$ is equal to $\psi_j$.  Then $\ddiv( \theta ) +D$ contains a point $v \in \Gamma_j$ if and only if $v$ is in either
\begin{enumerate}
\item  the divisor $\ddiv( \psi_j ) + D$, or
\item  the boundary of $\Gamma_j$.
\end{enumerate}
\end{ShapeLemma}

\noindent This shape lemma for minima is combined with another lemma about shapes of canonical divisors to reach the contradiction that proves the Gieseker--Petri theorem in \cite{tropicalGP}.

\section{Max Noether's theorem}

Here we examine functions in the canonical and 2-canonical linear series using trivalent and 3-edge-connected graphs.  This section is not logically necessary for the proof of Theorem~\ref{thm:mainthm}, and can be safely skipped by a reader who is interested only in the proof of the maximal rank conjecture for quadrics.  Nevertheless, the two are not unrelated and we include this section because, as explained in the introduction, Noether's theorem is a strong form of one case of the maximal rank conjecture for quadrics.  Also, the arguments presented here illustrate the potential for applying our methods to the study of linear series and multiplication maps using skeletons other than a chain of loops, which may be important for future work.

\medskip

Our arguments in this section depend on a careful analysis of the loci where piecewise linear functions attain their minima.
Recall that, for a divisor $D$ on a metric graph $\Gamma$, the tropical linear series $R(D)$ is the set of piecewise linear functions with integer slope $\psi$ on $\Gamma$ such that $\ddiv(\psi) + D$ is effective.  The tropical linear series $R(D)$ is a tropical module, which means that it is closed under scalar addition and pointwise minimum \cite[Lemma~4]{HMY12}.  For $v \in \Gamma$, we write $\deg_v(D)$ for the coefficient of $v$ in the divisor $D$, and for  a piecewise linear function $\psi$, we write
$$ \Gamma_\psi = \{ v \in \Gamma \ \vert \ \psi(v) = \min_{w \in \Gamma} \psi(w) \} $$
for the subgraph on which $\psi$ attains its global minimum.

\begin{lemma}
\label{Lemma:Outdegree}
Let $D$ be a divisor on $\Gamma$ with $\psi \in R(D)$.  Then, for any point $v \in \Gamma_\psi$,
$$ \outdeg_{\Gamma_\psi} (v) \leq \deg_v (D) . $$
\end{lemma}

\begin{proof}
Since $\psi$ obtains its minimum value at $v$, all of the outgoing slopes of $\psi$ at $v$ are nonnegative, and those along edges that are not in $\Gamma_\psi$ are strictly positive.  Since all of these slopes are integers and $\ddiv (\psi) + D$ is effective, it follows that $\outdeg_{\Gamma_\psi}(v)$ is at most $\deg_v(D)$.
\end{proof}

Recall that the canonical divisor $K_\Gamma$ is given by
\[
\deg_v(K_\Gamma) = \val(v) - 2,
\]
where $\val(v)$ is the valence (or number of outgoing edges) of $v$ in $\Gamma$. The following lemma restricts the loci on which functions in $R(K_\Gamma)$ attain their minimum.

\begin{lemma}
Let $\psi$ be a piecewise linear function in $R(K_\Gamma)$.  Then the subgraph $\Gamma_\psi$ on which $\psi$ attains its minimum is a union of edges in $\Gamma$ and has no leaves.
\end{lemma}

\begin{proof}
By Lemma~\ref{Lemma:Outdegree}, the outdegree $\outdeg_v(\Gamma_\psi)$ is at most $\deg(v)-2$ at each point $v \in \Gamma_\psi$.  It follows that any edge which contains a point of $\Gamma_\psi$ in its interior is entirely contained in $\psi$, and the number of edges in $\Gamma_\psi$ containing any vertex $v$ is at least two, so $\Gamma_\psi$ has no leaves.
\end{proof}

As a first application, we show that every loop in $\Gamma$ is the locus where some function in $R(K_\Gamma)$ attains its minimum, and that this function lifts to a canonical section on any totally degenerate curve whose skeleton is $\Gamma$.  Here, a \emph{loop} is an embedded circle in $\Gamma$ or, equivalently, a connected subgraph in which every point has valence 2.

\begin{proposition}
\label{Prop:CanonicalSections}
Let $\Gamma$ be a metric graph and let $\Gamma' \subset \Gamma$ be a loop.  Then there is a function $\psi \in R(K_{\Gamma})$ such that the subgraph $\Gamma_\psi$ on which $\psi$ attains its minimum is exactly $\Gamma'$.

Furthermore, if $X$ is a smooth projective curve over a nonarchimedean field such that the minimal skeleton of the Berkovich analytic space $X^{\an}$ is isometric to $\Gamma$, and $K_X$ is a canonical divisor that tropicalizes to $K_\Gamma$, then $\psi$ can be chosen to be $\trop(f)$ for some $f \in \cL(K_X)$.
\end{proposition}

\begin{proof}
Let $g$ be the first Betti number of $\Gamma$.  Choose points $v_1 , \ldots , v_{g-1}$ of valence 2 in $\Gamma \smallsetminus \Gamma'$ such that $\Gamma \smallsetminus \{ v_1 , \ldots , v_{g-1} \}$ is connected.  Since $K_\Gamma$ has rank $g-1$, there is a divisor $D \sim K_{\Gamma}$ such that $D - v_1 - \cdots - v_{g-1}$ is effective.  Let $\psi$ be a piecewise linear function such that $K_\Gamma + \ddiv(\psi) = D$.

By Lemma~\ref{Lemma:Outdegree}, the subgraph $\Gamma_\psi \subset \Gamma$ where $\psi$ attains its minimum is a union of edges of $\Gamma$ and has no leaves.  Since $\ord_{v_i} (\psi)$ is positive for $1 \leq i \leq g-1$, it follows that $\Gamma_\psi$ does not contain any of the points $v_1, \ldots, v_{g-1}$.  Being a subgraph of $\Gamma \smallsetminus \{v_1, \ldots, v_{g-1} \}$, the first Betti number of $\Gamma_\psi$ is at most 1.  On the other hand, every point has valence at least two in $\Gamma_\psi$.  It follows that $\Gamma_\psi$ is a loop, and hence must be the unique loop $\Gamma'$ contained in $\Gamma \smallsetminus \{v_1, \ldots, v_{g-1} \}$.

We now prove the last part of the proposition.   Let $p_1 , \ldots , p_{g-1}$ be points in $X$ specializing to $v_1 , \ldots , v_{g-1}$, respectively.  Since $K_X$ has rank $g-1$, there is a rational function $f \in \cL (K_X)$ such that $\ddiv (f) + K_\Gamma - p_1 - \cdots - p_{g-1}$ is effective.  From this we see that $\ddiv (\trop (f)) + K_{\Gamma} - v_1 -  \cdots - v_{g-1}$ is effective, and the proposition follows.
\end{proof}

Our next lemma controls the locus where a piecewise linear function in $R(2K_\Gamma)$ attains its minimum, when $\Gamma$ is trivalent.

\begin{lemma}
\label{Lemma:2Canonical}
Suppose $\Gamma$ is trivalent and let $\psi$ be a piecewise linear function in $R(2K_{\Gamma})$.  Then $\Gamma_\psi$ is a union of edges in $\Gamma$.
\end{lemma}

\begin{proof}
If $v \in \Gamma_\psi$ lies in the interior of an edge of $\Gamma$ then, by Lemma \ref{Lemma:Outdegree}, we have $\outdeg_{\Gamma_\psi}(v) = 0$, so $\Gamma_\psi$ contains the entire edge.  On the other hand, if $v \in \Gamma_\psi$ is a trivalent vertex of $\Gamma$ then Lemma \ref{Lemma:Outdegree} says that $\outdeg_{\Gamma_\psi} (v) \leq 2$.  It follows that $\Gamma_{\psi}$ contains at least one of the three edges adjacent to $v$.
\end{proof}

We conclude this section by applying this lemma and the preceding proposition together with Menger's theorem to prove Theorem~\ref{Thm:TropicalMaxNoether}, the analogue of Noether's theorem for trivalent 3-connected graphs.

\begin{remark}
A similar application of Menger's theorem is used to prove an analogue of Noether's theorem for graph curves in \cite[\S4]{BayerEisenbud91}.
\end{remark}

\begin{proof}[Proof of Theorems~\ref{Thm:TropicalMaxNoether} and \ref{Thm:MaxNoether}]
Assume $\Gamma$ is trivalent and 3-edge-connected.  Let $e \subset \Gamma$ be an edge with endpoints $v$ and $w$.  Since $\Gamma$ is 3-edge-connected, Menger's theorem says that there are two distinct paths from $v$ to $w$ that do not share an edge and do not pass through $e$.  Equivalently, there are two loops $\Gamma_1$ and $\Gamma_2$ in $\Gamma$ such that $\Gamma_1 \cap \Gamma_2 = e$.  By Proposition~\ref{Prop:CanonicalSections} there are functions $\psi_1$ and  $\psi_2$ in $R(K_{\Gamma})$ such that $\Gamma_{\psi_i} = \Gamma_i$.  We write $\psi_e = \psi_1 + \psi_2$, which is a piecewise linear function in $R(2K_{\Gamma})$.  Note that $\Gamma_{\psi_e} = e$.  Furthermore, again by Proposition~\ref{Prop:CanonicalSections}, if $X$ is a curve with skeleton $\Gamma$ and $K_X$ is a canonical divisor tropicalizing to $K_\Gamma$, then we can choose $f_1$ and $f_2$ in $\cL(K_X)$ such that $\psi_i = \trop(f_i)$, and hence $\psi_e = \trop(f_1 \cdot f_2)$ is the tropicalization of a function in the image of $\mu_2 : \Sym^2(\cL(K_X)) \rightarrow \cL(2K_X)$.

We claim that the set of $3g-3$ functions $\{ \psi_e \}_e$ is tropically independent.  Suppose not.  Then there are constants $b_e$ such that $\min_e \{ \psi_e + b_e \}$ occurs twice at every point of $\Gamma$.  Let
\[
\theta = \min_{e} \{\psi_e + b_e \},
\]
which is a piecewise linear function in $R(2K_\Gamma)$.  By Lemma~\ref{Lemma:2Canonical}, the function $\theta$ achieves its minimum along an edge, and hence there must be two functions in the set $\{\psi_e + b_e \}_e$ that achieve their minima along this edge.  However, by construction, the functions $\psi_e + b_e$ achieve their minima along distinct edges, which is a contradiction. We conclude that $\{ \psi_e \}_e$ is tropically independent, as claimed.
\end{proof}

\section{Special divisors on a chain of loops}
\label{Section:TheGraph}

For the remainder of the paper, we focus our attention on a chain of loops with bridges $\Gamma$, as pictured in Figure \ref{Fig:TheGraph}.  Here, we briefly recall the classification of special divisors on $\Gamma$ from \cite{tropicalBN}, along with the characterization of vertex avoiding classes and their basic properties.

The graph $\Gamma$ has $2g+2$ vertices, one on the lefthand side of each bridge, which we label $w_0, \ldots , w_g$, and one on the righthand side of each bridge, which we label $v_1, \ldots, v_{g+1}$.  There are two edges connecting the vertices $v_k$ and $w_k$, the top and bottom edges of the $k$th loop, whose lengths are denoted $\ell_k$ and $m_k$, respectively, as shown schematically in Figure~\ref{Fig:TheGraph}.
\begin{figure}[h!]
\begin{tikzpicture}

\draw (-2.7,-.5) node {\footnotesize $w_0$};
\draw [ball color=black] (-2.5,-0.3) circle (0.55mm);
\draw (-2.5,-.3)--(-1.93,-0.25);
\draw [ball color=black] (-1.93,-0.25) circle (0.55mm);
\draw (-1.95,-0.5) node {\footnotesize $v_1$};
\draw (-1.5,0) circle (0.5);
\draw (-1,0)--(0,0.5);
\draw [ball color=black] (-1,0) circle (0.55mm);
\draw (-0.85,0.3) node {\footnotesize $w_1$};
\draw (0.7,0.5) circle (0.7);
\draw (1.4,0.5)--(2,0.3);
\draw [ball color=black] (1.4,0.5) circle (0.55mm);
\draw [ball color=black] (0,0.5) circle (0.55mm);
\draw (-0.2,0.75) node {\footnotesize $v_2$};
\draw (2.6,0.3) circle (0.6);
\draw (3.2,0.3)--(3.87,0.6);
\draw [ball color=black] (2,0.3) circle (0.55mm);
\draw [ball color=black] (3.2,0.3) circle (0.55mm);
\draw [ball color=black] (3.87,0.6) circle (0.55mm);
\draw (4.5,0.3) circle (0.7);
\draw (5.16,0.5)--(5.9,0);
\draw (6.4,0) circle (0.5);
\draw [ball color=black] (5.16,0.5) circle (0.55mm);
\draw (5.48,0.74) node {\footnotesize $w_{g-1}$};
\draw [ball color=black] (5.9,0) circle (0.55mm);
\draw [ball color=black] (6.9,0) circle (0.55mm);
\draw (5.7,-.2) node {\footnotesize $v_g$};
\draw (7.15,-.2) node {\footnotesize $w_g$};
\draw (6.9,0)--(7.5,0.2);
\draw [ball color=black] (7.5,0.2) circle (0.55mm);
\draw (7.94,0.1) node {\footnotesize $v_{g+1}$};

\draw [<->] (3.35,0.475)--(3.8,0.7);
\draw [<->] (3.3,0.4) arc[radius = 0.715, start angle=10, end angle=170];
\draw [<->] (3.3,0.2) arc[radius = 0.715, start angle=-9, end angle=-173];

\draw (3.5,0.725) node {\footnotesize$n_k$};
\draw (2.5,1.25) node {\footnotesize$\ell_k$};
\draw (2.75,-0.7) node {\footnotesize$m_k$};
\end{tikzpicture}
\caption{The graph $\Gamma$.}
\label{Fig:TheGraph}
\end{figure}
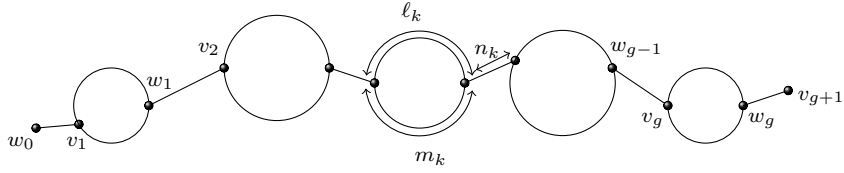
For $1 \leq k \leq g+1$ there is a bridge connecting $w_k$ and $v_{k+1}$, which we refer to as the $k$th bridge $\beta_k$, of length $n_k$.  Throughout, we assume that $\Gamma$ has admissible edge lengths in the following sense, which is stronger than the genericity conditions in \cite{tropicalBN, tropicalGP}.

\begin{definition} \label{Def:Admissible}
The graph $\Gamma$ has \emph{admissible edge lengths} if
$$4g m_k < \ell_k \ll \min \{ n_{k-1}, n_k\} \mbox{ for all $k$},$$ and there are no nontrivial linear relations $c_1 m_1 + \cdots + c_g m_g = 0$ with integer coefficients of absolute value at most $g + 1$.
\end{definition}

\begin{remark}
The inequality $4gm_k < \ell_k$ is required to ensure that the shapes of the functions $\psi_i$ and $\psi_{ij}$ are as described in \S\ref{Sec:Permissible}-\ref{Sec:Excess}.  Both inequalities are used in the proof of Lemma~\ref{Lem:LongBridges}, and the required upper bound on $\ell_k$ depends on the size of the multisets.  For multisets of size $m$, we assume $2m \ell_k < \min \{ n_{k-1}, n_k \}$.  In particular, for Theorem~\ref{thm:mainthm}, the inequality $4 \ell_k <  \min \{ n_{k-1}, n_k \}$ would suffice.  The condition on integer linear relations is used in the proof of Proposition~\ref{Prop:ConsecutiveBridges}.
\end{remark}

The special divisor classes on a chain of loops, i.e. the classes of effective divisors $D$ such that $r(D) > \deg(D) - g$, are explicitly classified in \cite{tropicalBN}.  Every effective divisor on $\Gamma$ is equivalent to an effective $w_0$-reduced divisor $D_0$, which has $d_0$ chips at the vertex $w_0$, together with at most one chip on every loop.  We may therefore associate to each equivalence class the data $(d_0 , x_1 , x_2 , \ldots x_g )$, where $x_i \in \mathbb{R} / (\ell_i + m_i ) \mathbb{Z}$ is the distance from $v_i$ to the chip on the $i$th loop in the counterclockwise direction, if such a chip exists, and $x_i = 0$ otherwise.  The associated lingering lattice path in $\ZZ^r$, whose coordinates we number from $0$ to $r-1$, is a sequence of points $p_0, \ldots, p_g$ starting at
\[
p_0 = (d_0, d_0 - 1, \ldots, d_0 - r + 1),
\]
 with the $i$th step given by
\begin{displaymath}
p_i - p_{i-1} = \left\{ \begin{array}{ll}
(-1,-1, \ldots , -1) & \textrm{if $x_i = 0$},\\
e_j & \textrm{if $x_i = (p_{i-1}(j) +1)m_i$ mod $\ell_i + m_i$} \\
  & \textrm{and both $p_{i-1}$ and $p_{i-1} + e_j$ are in $\mathcal{C}$},\\
0 & \textrm{otherwise},
\end{array} \right.
\end{displaymath}
where $e_0 , \ldots e_{r-1}$ are the basis vectors in $\mathbb{Z}^r$ and $\mathcal{C}$ is the open Weyl chamber
$$ \mathcal{C} = \{ y \in \mathbb{Z}^r \ \vert \ y_0 > \cdots > y_{r-1} > 0 \}. $$
By \cite[Theorem 4.6]{tropicalBN}, a divisor $D$ on $\Gamma$ has rank at least $r$ if and only if the associated lingering lattice path lies entirely in the open Weyl chamber $\mathcal{C}$.

The steps in the direction 0 are referred to as \textit{lingering steps}, and the number of lingering steps cannot exceed the Brill--Noether number $\rho (g,r,d)$.  In the case where $\rho (g,r,d) = 0$, such lattice paths are in bijection with rectangular tableaux of size $(r+1) \times (g-d+r)$.  This bijection is given as follows.  We label the columns of the tableau from $0$ to $r$ and place $i$ in the $j$th column when the $i$th step is in the direction $e_j$, and we place $i$ in the last column when the $i$th step is in the direction $(-1,\ldots,-1)$.

An open dense subset of the special divisor classes of degree $d$ and rank $r$ on $\Gamma$ are \emph{vertex avoiding}, in the sense of \cite[Definition~2.3]{LiftingDivisors}, which means that
\begin{itemize}
\item  the associated lingering lattice path has exactly $\rho(g,r,d)$ lingering steps,
\item  for any $i$, $x_i \neq m_i$ mod $(\ell_i + m_i)$, and
\item  for any $i$ and $j$, $x_i \neq (p_{i-1}(j))m_i$ mod $(\ell_i + m_i)$.
\end{itemize}
Vertex avoiding classes come with a useful collection of canonical representatives.  If $D$ is a divisor of rank $r$ on $\Gamma$ whose class is vertex avoiding, then there is a unique effective divisor $D_i \sim D$ such that $\deg_{w_0}(D_i) = i$ and $\deg_{v_{g+1}}(D_i) = r-i$.  Equivalently, $D_i$ is the unique divisor equivalent to $D$ such that $D_i - iw_0 - (r-i)v_{g+1}$ is effective.   Furthermore,
\begin{itemize}
\item  the divisor $D_i$ has no points on any of the bridges,
\item  for $i<r$, the divisor $D_i$ fails to have a point on the $j$th loop if and only if the $j$th step of the associated lingering lattice path is in the direction $e_i$,
\item  the divisor $D_r$ fails to have a point on the $j$th loop if and only if the $j$th step of the associated lingering lattice path is in the direction $(-1, \ldots , -1)$.
\end{itemize}

\begin{notation}
Throughout, we let $X$ be a smooth projective curve of genus $g$ whose analytification has skeleton $\Gamma$.  For the remainder of the paper, we let $D$ be a $w_0$-reduced divisor on $\Gamma$ of degree $d$ and rank $r$ whose class is vertex avoiding, $D_X$ a lift of $D$ to $X$, and $\psi_i$ a piecewise linear function on $\Gamma$ such that $D + \ddiv (\psi_i) = D_i$. By a \emph{lift} of $D$ to $X$, we mean that $D_X$ is a divisor of degree $d$ and rank $r$ on $X$ whose tropicalization is $D$.
\end{notation}

Note that $\psi_i$ is uniquely determined up to an additive constant, and for $i<r$ the slope of $\psi_i$ along the bridge $\beta_j$ is $p_j(i)$.   In this context, being $w_0$-reduced means that $D = D_r$, so the function $\psi_r$ is constant.  In particular, the functions $\psi_0, \ldots, \psi_r$ have distinct slopes along bridges, so $\{ \psi_0, \ldots, \psi_r \}$ is tropically independent.  For convenience, we set $p_j(r) = 0$ for all $j$.  

\medskip

\begin{proposition} \label{Prop:LiftFunctions}
There is a rational function $f_i \in \cL(D_X)$ such that $\trop(f_i) = \psi_i$.
\end{proposition}

\begin{proof}
The proof is identical to the proof of \cite[Proposition~6.5]{tropicalGP}, which is the special case where $\rho(g,r,d) = 0$.  \end{proof}

\noindent When $\rho(g,r,d) = 0$, all divisor classes of degree $d$ and rank $r$ are vertex avoiding.  Note that, since $\{ \psi_0, \ldots, \psi_r \}$ is tropically independent of size $r+1$, the set of rational functions $\{ f_0, \ldots, f_r \}$ is a basis for $\cL(D_X)$.

For a multiset $I \subset \{ 0, \ldots, r \}$ of size $m$, let $D_I = \sum_{i \in I} D_i$ and let $\psi_I$ be a piecewise linear function such that $mD + \ddiv \psi_I = D_I$.    By construction, the function $\psi_I$ is in $R(mD)$ and agrees with $\sum_{i \in I} \psi_i$ up to an additive constant.

\begin{conjecture} \label{Conj:TropicalMRC}
Suppose $r \geq 3$, $\rho(g,r,d) \geq 0$, and $d < g + r$.  Then there is a divisor $D$ of rank $r$ and degree $d$ whose class is vertex avoiding on a chain of loops $\Gamma$ with generic edge lengths, and a tropically independent subset $\cA \subset \{ \psi_I \ | \ \# I = m \}$ of size
\[
\# \cA = \min \left \{ {r + m \choose m}, \ md - g + 1 \right \}.
\]
\end{conjecture}

\noindent The conjecture is trivial for $r = 0$ and easy for $r = 1$, since the functions $k \psi_0$ have distinct nonzero slopes on every bridge and hence $\{0, \psi_0, \ldots, m \psi_0 \}$ is tropically independent.  Yet another easy case is $m = 1$, since $\{ \psi_0, \ldots \psi_r \}$ is tropically independent.   In the remainder of the paper we prove the conjecture for $m = 2$ and for $md < 2g + 4$.

\begin{proposition}
For any fixed $g$, $r$, $d$, and $m$, the maximal rank conjecture follows from Conjecture~\ref{Conj:TropicalMRC}.
\end{proposition}

\begin{proof}
Choose a smooth projective curve $X$ over a nonarchimedean field whose skeleton is $\Gamma$.  Then $X$ is Brill--Noether--Petri general \cite{tropicalGP} and $D$ lifts to a divisor $D_X$ of degree $d$ and rank $r$ on $X$ \cite{LiftingDivisors}.  We may assume $r \geq 1$, and it follows that $mD_X$ is nonspecial for $m \geq 2$ by Remark~\ref{Rem:Nonspecial}.  By Lemma~\ref{Lem:MultiplicationMaps}, the rank of $\mu_m$ is at least as large as any set $\cA$ such that $\{ \psi_I \ | \ I \in \cA \}$ is tropically independent.  Therefore, Conjecture~\ref{Conj:TropicalMRC} implies that $\mu_m : \Sym^m \cL(D_X) \rightarrow \cL(mD_X)$ has maximal rank and, as discussed in \S \ref{sec:prelim}, the maximal rank conjecture for $g$, $r$, $d$, and $m$ follows.
\end{proof}

\section{Two points on each loop}

In this section, we show that any nontrivial tropical dependence among the piecewise linear functions $\psi_I$, for multisets $I$ of size $m$, gives rise to a divisor equivalent to $mD$ with degree at least 2 at $w_0$, degree at least 2 at $v_{g+1}$, and degree at least 2 on each loop.  As a consequence, we deduce Theorem~\ref{Thm:LowDegree}, which confirms Conjecture~\ref{Conj:TropicalMRC} and the maximal rank conjecture for $md < 2g + 4$.

We begin with the following observation.

\begin{lemma}
\label{Lem:FunctionsDisagree}
Let $I$ and $J$ be distinct multisets of size $m$.  Then, for each loop $\gamma^\circ$ in $\Gamma$, the restrictions $D_I |_{\gamma^\circ}$ and $D_J|_{\gamma^\circ}$ are distinct.
\end{lemma}

\begin{proof}
Suppose $\gamma^{\circ}$ is the $j$th loop.  Let $q_i$ be the point on $\gamma^{\circ}$ whose distance from $v_j$ in the counterclockwise direction is $x_j - p_{j-1}(i) m_j$.  Then the degree of $q_i$ in $D_I$ is equal to the multiplicity of $i$ in the multiset $I$, unless the $j$th step of the lingering lattice path is in the direction $e_i$, in which case the degree of $q_i$ in $D_I$ is zero.  It follows that the multiset $I$ can be recovered from the restriction $D_I \vert_{\gamma^{\circ}}$.
\end{proof}

Let $\theta$ be the piecewise linear function
\[
\theta = \min_I \{ \psi_I \},
\]
which is in $R(mD)$, and let $\Delta$ be the corresponding effective divisor
\[
\Delta = mD + \ddiv \theta.
\]
By Lemma~\ref{Lem:FunctionsDisagree}, no two functions $\psi_I$ can agree on an entire loop, so if the minimum occurs everywhere at least twice on a loop, then there are at least three functions $\psi_I$ that achieve the minimum at some point of the loop.  We will study $\theta$ and $\Delta$ by systematically using observations like this one, examining behavior on each piece of $\Gamma$ and controlling which functions $\psi_I$ can achieve the minimum at some point in each loop.

Recall that, for $0 \leq k \leq g$, the $k$th bridge $\beta_k$ connects $w_k$ to $v_{k+1}$.  Let $u_k$ be the midpoint of $\beta_{k-1}$.  We then decompose $\Gamma$ into $g+2$ locally closed subgraphs $\gamma_0, \ldots, \gamma_{g+1}$, as follows.  The subgraph $\gamma_0$ is the half-open interval $[w_0, u_1)$.  For $1 \leq i \leq g$, the subgraph $\gamma_i$, which includes the $i$th loop of $\Gamma$, is the union of the two half-open intervals $[u_i, u_{i+1})$, which contain the top and bottom edges of the $i$th loop, respectively.  Finally, the subgraph $\gamma_{g+1}$ is the closed interval $[u_{g+1}, v_{g+1}]$.  We further write $\gamma_i^\circ$ for the $i$th embedded loop in $\Gamma$, which is a closed subset of $\gamma_i$, for $1 \leq i \leq g$.  The decomposition
\[
\Gamma = \gamma_0 \sqcup \cdots \sqcup \gamma_{g+1}
\]
is illustrated by Figure \ref{Fig:Decomposition}.

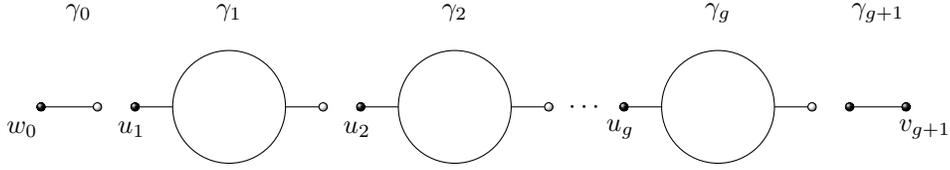
\begin{figure}[h!]
\begin{tikzpicture}
\matrix[column sep=0.5cm] {

\begin{scope}[grow=right, baseline]

\draw (-6.75, -.3) node {$w_0$};
\draw [ball color=black] (-6.5,0) circle (0.55mm);
\draw (-6.5,0)--(-5.75,0);
\draw (-6,1.25) node {$\gamma_0$};
\draw [ball color=white] (-5.75,0) circle (0.55mm);

\draw (-5.3, -.3) node {$u_1$};
\draw [ball color=black] (-5.25,0) circle (0.55mm);
\draw (-5.25,0)--(-4.75,0);
\draw (-4,0) circle (.75);
\draw (-4,1.25) node {$\gamma_1$};
\draw (-3.25,0)--(-2.75,0);
\draw [ball color=white] (-2.75,0) circle (0.55mm);

\draw [ball color=black] (-2.25,0) circle (0.55mm);
\draw (-2.3, -.3) node {$u_2$};
\draw (-2.25,0)--(-1.75,0);
\draw (-1,0) circle (.75);
\draw (-1,1.25) node {$\gamma_2$};
\draw (-.25,0)--(.25,0);
\draw [ball color=white] (.25,0) circle (0.55mm);

\draw(.75, 0) node {$\cdots$};

\draw [ball color=black] (1.25,0) circle (0.55mm);
\draw (1.2, -.3) node {$u_g$};
\draw (1.25,0)--(1.75,0);
\draw (2.5,0) circle (.75);
\draw (2.5,1.25) node {$\gamma_g$};
\draw (3.25,0)--(3.75,0);
\draw [ball color=white] (3.75,0) circle (0.55mm);

\draw [ball color=black] (4.25,0) circle (0.55mm);
\draw (4.25, 0)--(5,0);
\draw (4.625,1.25) node {$\gamma_{g+1}$};
\draw [ball color=black] (5,0) circle (0.55mm);
\draw (5.25, -.3) node {$v_{g+1}$};

\end{scope}

\\};
\end{tikzpicture}
\caption{Decomposition of the graph $\Gamma$ into locally closed pieces $\{\gamma_k\}$.}
\label{Fig:Decomposition}
\end{figure}

\begin{proposition}
\label{Prop:TwoChipsPerLoop}
Suppose the minimum of $\{ \psi_I (v) \}_I$ occurs at least twice at every point $v$ in $\Gamma$.  Then $\deg_{w_0}(\Delta)$, $\deg_{v_{g+1}}(\Delta)$, and $\deg(\Delta|_{\gamma_i^\circ})$ are all at least $2$.
\end{proposition}

\begin{proof}
Note that exactly one function $\psi_I$ has slope $mr$ on the first bridge; this is the function corresponding to the multiset $I = \{ 0, \ldots, 0 \}$.  Similarly, the only multiset that gives slope $mr -1$ is $\{ 1, 0, \ldots, 0 \}$.  Therefore, if the minimum occurs twice along the first bridge, then the outgoing slope of $\theta$ at $w_0$ is at most $mr -2$, and hence $\deg_{w_0}(\Delta) \geq 2$, as required.  Similarly, we have $\deg_{v_{g+1}}(\Delta) \geq 2$.

It remains to show that $\deg(\Delta|_{\gamma_i^\circ}) \geq 2$ for $1 \leq i \leq g$.  Choose a point $v \in \gamma_i^\circ$. By assumption, there are at least two distinct multisets $I$ and $I'$ such that both $\psi_I$ and $\psi_{I'}$ obtain the minimum on some closed interval containing $v$.  By Lemma~\ref{Lem:FunctionsDisagree}, the functions $\psi_I$ and $\psi_{I'}$ do not agree on all of $\gamma_i^\circ$, so there is another point $v' \in \gamma_i^\circ$ where at least one of these two functions does not obtain the minimum.  Without loss of generality, assume that $\psi_I$ does not obtain the minimum at $v'$.  Then $\psi_I$ obtains the minimum on a proper closed subset of $\gamma_i^\circ$, and since $\gamma_i^\circ$ is a loop, this set has outdegree at least two.  By the shape lemma for minima (see \S\ref{sec:prelim}), it follows that $(\ddiv (\theta) + mD)|_{\gamma_i^\circ}$ has degree at least two.
\end{proof}

As an immediate application of this proposition, we prove Conjecture~\ref{Conj:TropicalMRC} for $md < 2g + 4$.

\begin{theorem}
\label{Thm:LowDegree}
If $md < 2g+4$ then $\{\psi_I \ | \ \# I = m \}$ is tropically independent.
\end{theorem}

\begin{proof}
Suppose that $\{ \psi_I \}_I$ is tropically dependent.  After adding a constant to each $\psi_I$, we may assume that the minimum $\theta(v) = \min_I \psi_I(v)$ occurs at least twice at every point $v$ in $\Gamma$.  By Proposition \ref{Prop:TwoChipsPerLoop}, the restriction of $\Delta = mD + \ddiv (\theta)$ to each of the $g+2$ locally closed subgraphs $\gamma_k \subset \Gamma$ has degree at least two.  Therefore the degree of $\Delta$ is at least $2g + 4$, and the theorem follows.
\end{proof}

\noindent In particular, the maximal rank conjecture holds for $md < 2g + 4$.  This partially generalizes the case where $m =2$ and $d < g + 2$, proved by Teixidor i Bigas \cite{Teixidor03}.  Note, however, that \cite{Teixidor03} proves that the maximal rank condition holds for \emph{all} divisors of degree less than $g+2$, whereas Theorem \ref{Thm:LowDegree} implies this statement only for a general divisor.

\section{Permissible functions}
\label{Sec:Permissible}

In the preceding section, we introduced a decomposition of $\Gamma$ as the disjoint union of locally closed subgraphs $\gamma_0, \ldots, \gamma_{g+1}$ and proved that if $\theta(v) = \min_I \psi_I(v)$ occurs at least twice at every point $v$ in $\gamma_i$ then the degree of $\Delta = mD + \ddiv (\theta)$ restricted to $\gamma_i$ is at least 2.  These degrees of restrictions $\Delta|_{\gamma_i}$ appear repeatedly throughout the rest of the paper, so we fix
\[
\delta_i = \deg(\Delta|_{\gamma_i}).
\]
By Proposition~\ref{Prop:TwoChipsPerLoop}, we have $\delta_i \geq 2$ for all $i$.

We now discuss how the nonnegative integer vector $\delta = (\delta_0, \ldots, \delta_{g+1})$ restricts the multisets $I$ such that $\psi_I$ can achieve the minimum on the $k$th loop of $\Gamma$.

For $a \leq b$, let $\Gamma_{[a,b]}$ be the locally closed, connected subgraph
\[
\Gamma_{[a,b]} = \gamma_a \sqcup \cdots \sqcup \gamma_b.
\]
Note that the degrees of divisors in a tropical linear series restricted to such subgraphs are governed by the slopes of the associated piecewise linear functions, as follows.

Suppose $\Gamma' \subset \Gamma$ is a closed connected subgraph and $\psi$ is a piecewise linear function with integer slopes on $\Gamma$.  Then $\ddiv (\psi|_{\Gamma'})$ has degree zero and the multiplicity of each boundary point $v \in \partial \Gamma'$ is the sum of the incoming slopes at $v$, along the edges in $\Gamma'$.  Now $\ddiv(\psi)|_{\Gamma'}$ agrees with $\ddiv(\psi|_{\Gamma'})$ except at the boundary points and a simple computation at the boundary points of the locally closed subgraph $\gamma_k$, for $ 1\leq k \leq g$ shows that
\[
\deg(\ddiv(\psi)|_{\gamma_k}) = \sigma_k(\psi) - \sigma_{k+1}(\psi),
\]
where $\sigma_k(\psi)$ is the incoming slope of $\psi$ from the left at $u_k$.  Similarly,
\[
\deg (\ddiv(\psi)|_{\Gamma_{[0,k]}}) = - \sigma_{k+1}(\psi).
\]
Our indexing conventions for lingering lattice paths are chosen for consistency with \cite{tropicalBN}, and with this notation we have
\[
\sigma_k(\psi_i) = p_{k-1}(i).
\]
These slopes, and the conditions on the edge lengths on $\Gamma$, lead to restrictions on the multisets $I$ such that $\psi_I$ achieves the minimum at some point in the $k$th loop $\gamma_k^\circ$.

\begin{definition}
Let $I \subset \{0, \ldots, r \}$ be a multiset of size $m$.  We say that $\psi_I$ is $\delta$\emph{-permissible} on $\gamma_k^\circ$ if
$$ \deg (D_I \vert_{\Gamma_{\leq k -1}}) \geq \delta_0 + \cdots + \delta_{k-1} $$
and
$$ \deg (D_I \vert_{\Gamma_{\leq k}}) \leq \delta_0 + \cdots + \delta_k . $$
We say that $\psi_I$ is $\delta$\emph{-permissible} on $\Gamma_{[a,b]}$ if there is some $k \in [a,b]$ such that $\psi_I$ is $\delta$-permissible on $\gamma_k^\circ$.
\end{definition}

\begin{lemma}
\label{Lem:LongBridges}
If $\psi_I (v) = \theta (v)$ for some $v \in \gamma_k^\circ$ then $\psi_I$ is $\delta$-permissible on $\gamma_k^\circ$.
\end{lemma}

\begin{proof}
Recall that the edge lengths of $\Gamma$ are assumed to be admissible, in the sense of Definition~\ref{Def:Admissible}.

Suppose $\psi_I(v) = \theta(v)$ for some point $v$ in $\gamma_k^\circ$.  We claim that the slope of $\psi_I$ along the bridge $\beta_{k-1}$ to the left of the loop is at most the incoming slope of $\theta$ from the left at $u_{k-1}$.  Indeed, if the slope of $\psi_I$ is strictly greater than that of $\theta$ then, since $\psi_I(u_{k-1}) \geq \theta(u_{k-1})$ and the slope of $\theta$ can only decrease when going from $u_{k-1}$ to $v_k$, the difference $\psi_I(v_k) - \theta(v_k)$ will be at least the distance from $u_{k-1}$ to $v_k$, which is $n_{k-1} / 2$.

The slopes of $\psi_I$ and $\theta$ along the bottom edge are between $0$ and $mg$, and the slopes along the top edge are between $0$ and $m$.  Since $\ell_k > 4 g m_k$ by assumption, it follows that $|\psi_I - \theta|$ changes by at most $m \ell _k$ between $v_k$ and any other point in $\gamma_k^\circ$.  Assuming $2m \ell_k < n_{k-1}$, this proves the claim.

Note that the incoming slopes of $\psi_I$ and $\theta$ from the left at $u_k$ are
$$ \deg ( mD \vert_{\Gamma_{[0,k-1]}}) - \deg ( D_I \vert_{\Gamma_{[0,k-1]}}), \mbox{ \  and \ } \deg ( mD \vert_{\Gamma_{[0,k-1]}}) - \delta_0 - \cdots - \delta_{k-1}, $$ respectively.  Therefore, the claim implies that $ \deg (D_I \vert_{\Gamma_{[0,k-1]}}) \geq \delta_0 + \cdots + \delta_{k-1}$.

A similar argument using slopes along the bridge $\beta_k$ to the right of $\gamma_k^\circ$ and assuming $2m \ell_k < n_k$ shows that $ \deg (D_I \vert_{\Gamma_{\leq k}}) \leq \delta_0 + \cdots + \delta_k,$ and the lemma follows.
\end{proof}

Our general strategy for proving Conjecture~\ref{Conj:TropicalMRC} is to choose the set $\cA$ carefully, assume that the minimum occurs everywhere at least twice, and then bound $\delta_0 + \ldots + \delta_i$ inductively, moving from left to right across the graph.  By induction, we assume a lower bound on $\delta_0 + \cdots + \delta_i$.  Then, for a carefully chosen $j > i$, we consider $\delta_0 + \cdots + \delta_{j}$.  If this is too small, then Lemma~\ref{Lem:LongBridges} severely restricts which functions $\psi_I$ can achieve the minimum on loops in $\Gamma_{[i+1, j]}$, making it impossible for the minimum to occur everywhere at least twice unless the bottom edge lengths $m_{i+1}, \ldots, m_{j}$ satisfy a nontrivial linear relation with small integer coefficients.  We deduce a lower bound  on $\delta_0 + \cdots + \delta_{j}$ and continue until we can show $\delta_0 + \cdots + \delta_{g+1} > 2d$, a contradiction.  We give a first taste of this type of argument in Lemma~\ref{Lem:ThreeFunctions} and Example~\ref{Ex:Genus10}.   Example~\ref{Ex:QuadricRelation} illustrates how similar techniques may be applied to understand the kernel of $\mu_m$ when it is not injective.    A more general (and more technical) version of the key step in this argument, using the assumption that a small number of functions $\psi_I$ achieve the minimum everywhere at least twice on $\Gamma_{[i+1, j]}$ to produce a nontrivial linear relation with small integer coefficients, appears in the proof of Proposition~\ref{Prop:ConsecutiveBridges}.

\begin{notation}
For the remainder, we fix $m = 2$, and $I$ and $I_j$ will always denote multisets of size $2$ in $\{0, \ldots, r\}$, which we identify with pairs $(i,j)$ with $0 \leq i \leq j \leq r$.  We write $\psi_{ij}$ for the piecewise linear function $\psi_i + \psi_j$ corresponding to the multiset $I = \{i,j\}$.
\end{notation}

\begin{lemma}
\label{Lem:ThreeFunctions}
Suppose that $\delta_k = 2$ and $\theta(v) = \min \{ \psi_{I_1}(v), \psi_{I_2}(v), \psi_{I_3}(v) \}$ occurs at least twice at every point in $\gamma_k^\circ$.  Then, $\theta|_{\gamma_k^\circ} = \psi_{I_j}|_{\gamma_k^\circ}$, for some $1 \leq j \leq 3$.
\end{lemma}

\begin{proof}
By Lemma~\ref{Lem:FunctionsDisagree}, no two of the functions may obtain the minimum on all of $\gamma_k^\circ$. After renumbering, we may assume that $\psi_{I_3}$ obtains the minimum on some but not all of the loop.  Let $v$ be a boundary point of the locus where $\psi_{I_3}$ obtains the minimum. Since there are only three functions that obtain the minimum, one must obtain the minimum in a neighborhood of $v$.  After renumbering we may assume that this is $\psi_{I_1}$.  We claim that $\theta$ is equal to $\psi_{I_1}$ on the whole loop.  If not, then by the shape lemma for minima, $D + \ddiv \theta$ would contain the two points in the boundary of the locus where $\psi_{I_1}$ obtains the minimum, in addition to $v$, contradicting the assumption that $\delta_k$, the degree of $D + \ddiv \theta$ on $\gamma_k$, is 2.
\end{proof}

\begin{remark}  \label{rem:essunique}
It follows from Lemma~\ref{Lem:ThreeFunctions} that, under the given hypotheses, the tropical dependence on the $k$th loop is essentially unique, in the sense that if $b_1$, $b_2$, and $b_3$ are real numbers such that
\[
\theta (v) = \min \{ \psi_{I_1}(v) + b_1, \psi_{I_2}(v) + b_2, \psi_{I_3}(v) + b_3\}
\]
occurs at least twice at every point on the $k$th loop, then $b_1 = b_2 = b_3$.  Furthermore, since each $\psi_{I_j}$ has constant slope along the bottom edge of $\gamma_k$ and no two agree on the entire top edge, there must be one pair that agrees on the full bottom edge and part of the top edge and another pair that agrees on part of the top edge, as shown in Figure~\ref{Fig:ThreeFunctions}.  Note that the divisor $D + \ddiv \theta$ consists of two points on the top edge and one (but not both) of these points may lie at one of the end points, $v_k$ or $w_k$.

\begin{figure}[h!]
\begin{tikzpicture}
\matrix[column sep=0.5cm] {

\begin{scope}[grow=right, baseline]
\draw (-1,0) circle (1);
\draw (-3,0)--(-2,0);
\draw (0,0)--(1,0);
\draw [ball color=black] (-0.29,0.71) circle (0.55mm);
\draw [ball color=black] (-1.71,0.71) circle (0.55mm);
\draw (-1,1.25) node {{\tiny $\psi_{I_1}\ \psi_{I_2}$}};
\draw (-1,-1.25) node {{\tiny $\psi_{I_1}\ \psi_{I_3}$}};
\end{scope}

\\};
\end{tikzpicture}

\caption{An illustration of the regions where different functions obtain the minimum in the situation of Lemma \ref{Lem:ThreeFunctions}.}
\label{Fig:ThreeFunctions}
\end{figure}
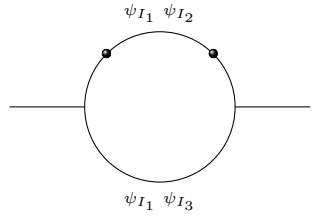

\end{remark}

Before we turn to the proof of the main theorem, we illustrate the techniques involved with a pair of examples.

\begin{example}
\label{Ex:Genus10}
Suppose $g = 10$, and let $D$ be the divisor of rank 4 and degree 12 corresponding to the tableau pictured in Figure \ref{Fig:Genus10}.  We note that this special case of the maximal rank conjecture for $m = 2$ is used to produce a counterexample to the slope conjecture in \cite{FarkasPopa05}.

\begin{figure}[h!]
\begin{tikzpicture}
\matrix[column sep=0.7cm, row sep = 0.7cm] {
\begin{scope}[node distance=0 cm,outer sep = 0pt]
	      \node[bsq] (11) at (2.5,1) {1};
	      \node[bsq] (21) [below = of 11] {2};
	      \node[bsq] (12) [right = of 11] {3};
	      \node[bsq] (22) [below = of 12] {4};
	      \node[bsq] (13) [right = of 12] {5};
	      \node[bsq] (23) [below = of 13] {6};
	      \node[bsq] (14) [right = of 13] {7};
	      \node[bsq] (24) [below = of 14] {8};
	      \node[bsq] (15) [right = of 14] {9};
	      \node[bsq] (25) [below = of 15] {10};
\end{scope}
\\};
\end{tikzpicture}
\caption{The tableau corresponding to the divisor $D$ in Example \ref{Ex:Genus10}.}
\label{Fig:Genus10}
\end{figure}

Assume that the minimum $\theta = \min \{ \psi_I \}$ occurs at least twice at every point of $\Gamma$.  By Proposition~\ref{Prop:TwoChipsPerLoop}, the divisor $\Delta = \ddiv (\theta ) + 2D$ has degree at least two on each of the 12 locally closed subgraphs $\gamma_k$.  Since $\deg (2D) = 24$, the degree of $\Delta$ on each of these subgraphs must be exactly 2.  In other words,  $\delta = (2, \ldots, 2)$.

In the lingering lattice path for $D$, we have
\begin{align*}
p_4 & = (6,5,2,1,0), \\
p_5 & = (6,5,3,1,0), \\
p_6 & = (6,5,4,1,0).
\end{align*}
The $\delta$-permissible functions $\psi_{ij}$ on $\Gamma_{[5,6]}$ are those such that either
\begin{align*}
 p_4 (i) + p_4 (j) & \leq 6  \mbox{ \  and \ }  p_5 (i) + p_5 (j)  \geq 6, \mbox{ \  or \ } \\
p_5 (i) + p_5 (j) & \leq 6  \mbox{ \  and \ }  p_6 (i) + p_6 (j)   \geq 6 .
\end{align*}
There are only 3 such pairs:  $(0,4)$, $(1,3)$, and $(2,2)$.  By Lemma \ref{Lem:ThreeFunctions}, in order for the minimum to occur at least twice at every point of $\Gamma_{[5,6]}$, on each of the two loops there must be a single function $\psi_{ij}$ that obtains the minimum at every point.  A simple case analysis shows that for both loops this function must be $\psi_{13}$ and that $\psi_{04}$ must achieve the minimum on both bottom edges.  Let $q_5$ and $q_6$ be the points of $D$ on $\gamma_5$ and $\gamma_6$, respectively, as shown in Figure \ref{Fig:Genus10Min}.

\begin{figure}[h!]
\begin{tikzpicture}
\matrix[column sep=.5cm] {

\begin{scope}[grow=right, baseline]
\draw (-1,0) circle (1);
\draw (-3,0)--(-2,0);
\draw (0,0)--(1,0);
\draw [ball color=black] (-1.707,0.707) circle (0.55mm);
\draw [ball color=black] (-.134,0.5) circle (0.55mm);
\draw [ball color=white] (-.5,.866) circle (0.55mm);
\draw (-.9,.65) node {{\tiny $\psi_{22}, \psi_{13}$}};
\draw (0.5,-0.7) node {{\tiny $\psi_{13},  \psi_{04}$}};
\draw (-.25,1.05) node {{\footnotesize $q_5$}};
\draw (2,0) circle (1);
\draw (3,0)--(4,0);
\draw [ball color=black] (1.29,0.71) circle (0.55mm);
\draw [ball color=black] (2.866,0.5) circle (0.55mm);
\draw [ball color=white] (2.5,.866) circle (0.55mm);
\draw (2.15,.65) node {{\tiny $\psi_{22}, \psi_{13}$}};
\draw (2.75,1.05) node {{\footnotesize $q_6$}};
\end{scope}

\\};
\end{tikzpicture}
\caption{Regions of $\Gamma_{[5,6]}$ on which the functions obtain the minimum.}
\label{Fig:Genus10Min}
\end{figure}
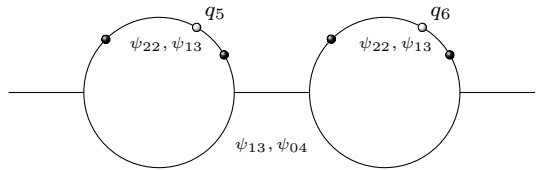

\noindent The regions of the graph are labeled by the pairs of functions $\psi_{ij} , \psi_{i'j'}$ that obtain the minimum on that region.
For each $i$, the change in value $\psi_i(q_6) - \psi_i(q_5)$ may be expressed as a function of the entries in the lattice path and the lengths of the edges in $\Gamma$.  Specifically, as we travel from $q_5$ to $q_6$, the slopes of $\psi_{22}$ and $\psi_{13}$ differ by 1 on an interval of length $m_5$ along the top edge of $\gamma_5$, and again on an interval of length $m_6$ along the top edge of $\gamma_6$.  This computation shows that
\[
(\psi_{22} (q_5) - \psi_{13} (q_5)) - (\psi_{22} (q_6) - \psi_{13} (q_6)) = m_5 - m_6.
\]
Since $\psi_{13}$ and $\psi_{22}$ agree at $q_5$ and $q_6$, it follows that $m_5$ must equal $m_6$, contradicting the hypothesis that $\Gamma$ has admissible edge lengths in the sense of Definition~\ref{Def:Admissible}.

We conclude that the minimum cannot occur everywhere at least twice, so $\{ \psi_I\}_I$ is tropically independent.  Therefore, for any curve $X$ with skeleton $\Gamma$ and any lift of $D$ to a divisor $D_X$ of rank $4$, the map
\[
\mu_2: Sym^2 \cL(D_X) \rightarrow \cL(2D_X)
\]
is injective.
\end{example}

We now consider an example illustrating our approach via tropical independence when $\mu_2$ is not injective.  Recall that the canonical divisor on a non-hyperelliptic curve of genus $4$ gives an embedding in $\PP^3$ whose image is contained in a unique quadric.  This is the special case of the maximal rank conjecture where $g$, $r$, $d$, and $m$ are $4$, $3$, $6$, and $2$, respectively.

\begin{example}  \label{Ex:QuadricRelation}
Suppose $g = 4$ and $m = 2$.  Note that the class of the canonical divisor $D = K_\Gamma$ is vertex avoiding of rank $3$.  Since $\Gamma$ is the skeleton of a curve whose canonical embedding lies on a quadric, the functions $\psi_I$ are tropically dependent, and we may assume $\min_I \psi_I(v)$ occurs at least twice at every point $v \in \Gamma$.

Let $\theta (v)  = \min_I \psi_I(v),$ and let $\Delta = 2K_\Gamma + \ddiv \theta$.  By Proposition~\ref{Prop:TwoChipsPerLoop}, the degree $\delta_k$ of $\Delta$ on $\gamma_k$ is at least 2 for $k = 0, \ldots, 5$.  Since $\deg(\Delta) = 12$, it follows that $\delta = (2, \ldots, 2)$.

The lingering lattice path associated to $K_{\Gamma}$ is given by
\begin{align*}
p_0 & = (3,2,1,0), \\
p_1 & = (4,2,1,0), \\
p_2 & = (4,3,1,0), \\
p_3 & = (4,3,2,0), \\
p_4 & = (3,2,1,0).
\end{align*}
Since $\delta_0 = \delta_1 = 2$, the $\delta$-permissible functions $\psi_{ij}$ on $\gamma_1$ are those such that
$$ p_0 (i) + p_0 (j) \leq 4 \mbox{ \  and \ } p_1 (i) + p_1 (j) \geq 4 . $$
There are only three such pairs:  $(0,2)$, $(1,1)$, and $(0,3)$.  In a similar way, we see that there are precisely three $\delta$-permissible functions on each loop $\gamma_k$.  By Lemma \ref{Lem:ThreeFunctions} and Remark~\ref{rem:essunique}, the tropical dependence among the three functions that achieve the minimum on each loop is essentially unique.  Figure~\ref{Fig:Genus4} illustrates the combinatorial structure of this dependence.
\begin{figure}[h!]
\begin{tikzpicture}
\matrix[column sep=0cm] {
\begin{scope}[grow=right, baseline]
\draw (-1,0) circle (1);
\draw (-2.7,0)--(-2,0);
\draw [ball color=black] (-2.7,0) circle (0.55mm);
\draw (0,0)--(1,0);
\draw [ball color=black] (-0.5,0.87) circle (0.55mm);
\draw [ball color=black] (0,0) circle (0.55mm);
\draw (-1.95,-1) node {{\tiny $\psi_{11}, \psi_{02}$}};
\draw (-.7,.4) node {{\tiny $\psi_{02}, \psi_{03}$}};
\draw (2,0) circle (1);
\draw (3,0)--(4,0);
\draw [ball color=black] (1,0) circle (0.55mm);
\draw [ball color=black] (2.71,0.71) circle (0.55mm);
\draw (0.5,-0.25) node {{\tiny $\psi_{11}, \psi_{03}$}};
\draw (1.75,1.2) node {{\tiny $\psi_{11}, \psi_{12}$}};
\draw (3.5,-.7) node {{\tiny $\psi_{12}, \psi_{03}$}};
\draw (5,0) circle (1);
\draw (6,0)--(7,0);
\draw [ball color=black] (6,0) circle (0.55mm);
\draw [ball color=black] (4.29,0.71) circle (0.55mm);
\draw (5.25,1.2) node {{\tiny $\psi_{12}, \psi_{22}$}};
\draw (8,0) circle (1);
\draw (9,0)--(9.7,0);
\draw [ball color=black] (7,0) circle (0.55mm);
\draw [ball color=black] (7.5,0.87) circle (0.55mm);
\draw (7.75,.4) node {{\tiny $\psi_{03}, \psi_{13}$}};
\draw (6.5,-0.25) node {{\tiny $\psi_{22}, \psi_{03}$}};
\draw (8.95,-1) node {{\tiny $\psi_{22}, \psi_{13}$}};
\draw [ball color=black] (9.7,0) circle (0.55mm);
\end{scope}
\\};
\end{tikzpicture}
\caption{The unique tropical dependence for the canonical linear system when $g = 4$ and $m = 2$.}
\label{Fig:Genus4}
\end{figure}
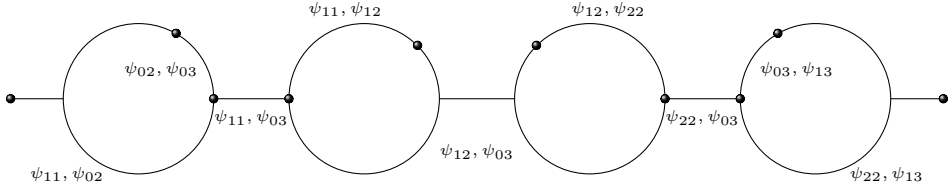

\noindent  Since this dependence among the functions that realize the minimum at some point in $\Gamma$ is essentially unique, omitting any one of the six functions that appear leaves a tropically independent set of size 9.  Therefore, the map
\[
\mu_2: Sym^2 \cL(D_X) \rightarrow \cL(2D_X)
\]
has rank at least 9.  Since $\cL(2D_X)$ has dimension $9$, it follows that $\mu_2$ is surjective.
\end{example}

\section{Shapes of functions, excess degree, and linear relations among edge lengths} \label{Sec:Excess}

In this section, and in \S\ref{sec:rhozero}, below, we assume that $\rho(g,r,d) = 0$.   All of the essential difficulties appear already in this special case.  The case where $\rho(g,r,d) > 0$ is treated in \S\ref{sec:rhopos} through a minor variation on these arguments.

\medskip

We now proceed with the more delicate and precise combinatorial arguments required to prove Theorem~\ref{thm:mainthm}.  With $g$, $r$, and $d$ fixed, and assuming $d - g \leq r$, we must produce a divisor $D$ of degree $d$ and rank $r$ on $\Gamma$, together with a set
\[
\cA \subset \{ (i,j) \ | \ 0 \leq i \leq j \leq r \}
\]
of size $$\# \cA = \min \left\{ {r+2 \choose 2}, 2d-g+1 \right \},$$ such that the corresponding collection of rational functions
\[
\{ \psi_{ij} \ | \ (i,j) \in \cA \}
\]
is tropically independent.

\begin{notation}
The quantity $g-d+r$ appears repeatedly throughout, so we fix the notation
\[
s = g-d+r,
\]
which simplifies various formulas.  The condition that $\rho(g,r,d) = 0$ means that $g = (r+1)s$.
\end{notation}

\noindent We now specify the divisor $D$ that we will use to prove Conjecture~\ref{Conj:TropicalMRC} for $m = 2$ when $\rho(g,r,d) = 0$.  The set $\cA$ is described in \S\ref{sec:rhozero}.

\begin{notation}  For the remainder of this section and \S\ref{sec:rhozero}, let $D$ be the divisor of degree $d$ and rank $r$ on $\Gamma$ corresponding to the standard tableau with $r+1$ columns and $s$ rows in which the numbers $1, \ldots, s$ appear in the leftmost column; $s+1, \ldots, 2s$ appear in the next column, and so on.  We number the columns from zero to $r$, so the $\ell$th column contains the numbers $\ell s + 1, \ldots, (\ell + 1) s$.  The specific case $g=10$, $r=4$, $d=12$ is illustrated in Figure~\ref{Fig:Genus10} from Example~\ref{Ex:Genus10}.
\end{notation}

\begin{remark}
Our choice of divisor is particularly convenient for the inductive step in the proof of Theorem~\ref{thm:mainthm}, in which we divide the graph $\Gamma$ into the $r +1$ regions $\Gamma_{[\ell s + 1, (\ell + 1)s]}$, for $0 \leq \ell \leq r$, and move from left to right across the graph, one region at a time, studying the consequences of the existence of a tropical dependence.  Since the numbers $\ell s + 1, \ldots, (\ell + 1) s$ all appear in the $\ell$th column, the slopes of the functions $\psi_i$, for $i \neq \ell$, are the same along all bridges and bottom edges, respectively, in the subgraph $\Gamma_{[\ell s + 1, (\ell + 1)s]}$.  Only the slopes of $\psi_\ell$ are changing in this region.
\end{remark}

Here we describe the \emph{shape} of the function $\psi_i$, by which we mean the combinatorial configuration of regions on the loops and bridges on which $\psi_i$ has constant slope, as well as the slopes from left to right on each region. These data determine (and are determined by) the combinatorial configurations of the points in $D_i = D + \ddiv(\psi_i)$.

Fix $ 0 \leq \ell \leq r$.  Suppose $\ell s + 1 \leq k \leq (\ell + 1)s$, so $\gamma_k^\circ$ is a loop in the subgraph $\Gamma_{[\ell s + 1, (\ell + 1) s]}$.  Recall from \S\ref{Section:TheGraph} and \S\ref{Sec:Permissible} that $D$ contains one point on the top edge of $\gamma_k^\circ$, at distance
\[
p_{k-1}(\ell) = \sigma_k(\ell)
\]
in the counterclockwise direction from $w_k$, where $\sigma_k(\ell)$ is the slope of $\psi_\ell$ along the bridge $\beta_k$.

\bigskip

\noindent \emph{Case 1:}  The shape of $\psi_i$, for $i < \ell$.  If $i < \ell$ then $D_i = D + \ddiv \psi_i$ contains one point on the top edge of $\gamma_k^\circ$, at distance $(r + s - i - 1 - \sigma_k(\ell)) \cdot m_k$ from $v_k$, the left endpoint of $\gamma_k^\circ$.  This is illustrated schematically in Figure~\ref{Fig:LeftHalf}.
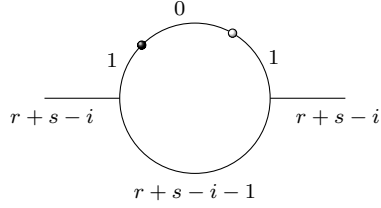
\begin{figure}[h!]
\begin{tikzpicture}
\matrix[column sep=0cm] {

\begin{scope}[grow=right, baseline]
\draw (-1,0) circle (1);
\draw (-3,0)--(-2,0);
\draw (0,0)--(1,0);
\draw [ball color=white] (-.5,.866) circle (0.55mm);
\draw [ball color=black] (-1.707,.707) circle (0.55mm);
\draw (-2.9,-0.25) node {\footnotesize $r+s-i$};
\draw (.9,-0.25) node {\footnotesize $r+s-i$};
\draw (-1,-1.25) node {\footnotesize $r+s-i-1$};
\draw (0.05,0.55) node {\footnotesize 1};
\draw (-2.1,0.5) node {\footnotesize 1};
\draw (-1.2,1.2) node {\footnotesize 0};
\end{scope}
&

\\};
\end{tikzpicture}
\caption{The shape of $\psi_i$, for $i<\ell$.}
\label{Fig:LeftHalf}
\end{figure}
The point of $D_i$ on the top edge of $\gamma_k^\circ$ is marked with a black circle, and the point of $D$ is marked with a white circle.  Each region of constant slope is labeled with the slope of $\psi_i$ from left to right.  The slope of $\psi_i$ from left to right along each bridge adjacent to $\gamma_k^\circ$ is $r + s - i$, and the slope along the bottom edge is $r + s - i -1$.

\bigskip

\noindent \emph{Case 2:}  The shape of $\psi_j$, for $j > \ell$.  If $j > \ell$ then $D_j  = D + \ddiv \psi_j$ contains one point on the top edge of $\gamma_k^\circ$, at distance $(\sigma_k(\ell) - r + j)$ from $w_k$, as shown in Figure~\ref{Fig:RightHalf}.
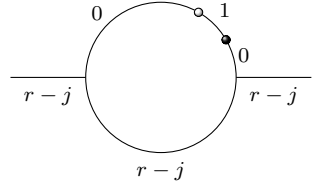
\begin{figure}[h!]
\begin{tikzpicture}
\matrix[column sep=0 cm] {

\begin{scope}[grow=right, baseline]
\draw (4,0) circle (1);
\draw (2,0)--(3,0);
\draw (5,0)--(6,0);
\draw [ball color=white] (4.5,.866) circle (0.55mm);
\draw [ball color=black] (4.866,0.5) circle (0.55mm);
\draw (2.5,-0.25) node {\footnotesize $r-j$};
\draw (5.5,-0.25) node {\footnotesize $r-j$};
\draw (4,-1.25) node {\footnotesize $r-j$};
\draw (3.15,0.85) node {\footnotesize 0};
\draw (5.1,0.3) node {\footnotesize 0};
\draw (4.85,.9) node {\footnotesize 1};
\end{scope}
&

\\};
\end{tikzpicture}
\caption{The shape of $\psi_j$, for $j>\ell$.}
\label{Fig:RightHalf}
\end{figure}
The slope of $\psi_j$ along the bottom edge and both adjacent bridges is $r - j$.

\bigskip

\noindent \emph{Case 3:}  The shape of $\psi_\ell$.  The divisor $D_\ell$ has no points on $\gamma_k^\circ$, as shown in Figure~\ref{Fig:NoHalf}.
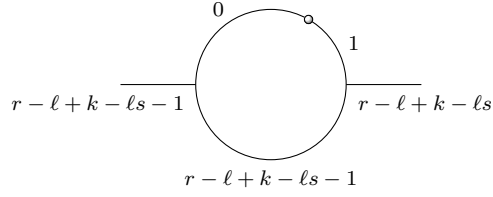
\begin{figure}[h!]
\begin{tikzpicture}
\matrix[column sep=0cm] {

\begin{scope}[grow=right, baseline]
\draw (9,0) circle (1);
\draw (7,0)--(8,0);
\draw (10,0)--(11,0);
\draw [ball color=white] (9.5,.866) circle (0.55mm);
\draw (6.7,-0.25) node {\footnotesize $r-\ell + k - \ell s -1$};
\draw (11.05,-0.25) node {\footnotesize $r-\ell + k - \ell s$};
\draw (9,-1.25) node {\footnotesize $r-\ell + k - \ell s -1$};
\draw (8.3,1) node {\footnotesize 0};
\draw (10.1,0.55) node {\footnotesize 1};
\end{scope}

\\};
\end{tikzpicture}
\caption{The shape of $\psi_\ell$ on $\gamma_k$.}
\label{Fig:NoHalf}
\end{figure}
Note that this is the only case in which the slope is not the same along the two bridges adjacent to $\gamma_k^\circ$.

\medskip

We use the shapes of the functions $\psi_i$ to control the set of pairs $(i,j)$ such that $\psi_{ij}$ is $\delta$-permissible on certain loops, as follows.  Suppose $\{ \psi_{ij} \}$ is tropically dependent, so there are constants $b_{ij}$ such that $\min \{ \psi_{ij}(v) + b_{ij} \}$ occurs at least twice at every point $v \in \Gamma$.  Replacing $\psi_{ij}$ with $\psi_{ij} + b_{ij}$, we may assume $\min \{ \psi_{ij}(v) \}$ occurs at least twice at every point.  Let
\[
\theta = \min_{ij} \{ \psi_{ij} \}, \mbox{ \ \ } \Delta = 2D + \ddiv(\theta), \mbox{ \ and \ } \delta_i = \deg(\Delta|_{\gamma_i}).
\]
By Proposition~\ref{Prop:TwoChipsPerLoop}, each $\delta_i$ is at least 2, and some may be strictly greater.  We keep track of the \emph{excess degree function}
\[
e(k) = \delta_0 + \cdots + \delta_k - 2k.
\]
It contains exactly the same information as $\delta$, but in a form that is somewhat more convenient for our inductive arguments in \S\ref{sec:rhozero}.  Note that $e(k)$ is positive and nondecreasing as a function of $k$.

In the induction step, we study the $\delta$-permissible functions $\psi_{ij}$ on subgraphs $$\Gamma_{[a(\ell), b(\ell)]} \subseteq \Gamma_{[\ell s + 1, (\ell + 1)s]},$$ where $a(\ell)$ and $b(\ell)$ are given by
\[
a(\ell) = \left\{ \begin{array}{ll} \ell s + 1 & \mbox{ for } \ell \leq \rr, \\ [3 pt]
								\ell (s +1) - \rr + 1 & \mbox{ for } \ell > \rr, \end{array} \right.
\]
and
\[
b(\ell) = \left\{ \begin{array}{ll} \ell (s + 1) - \rr + s & \mbox{ for } \ell \leq \rr, \\ [3 pt]
								(\ell + 1)s & \mbox{ for } \ell > \rr. \end{array} \right.
\]
Note that the subgraph $\Gamma_{[a(\ell), b(\ell)]}$ is only well-defined if $a(\ell) \leq b(\ell)$.  This is the case when $\ell$ is in the range
$$ \max \left\{ 0, \rr - s \right\} \leq \ell < \min \left\{ r, \rr + s \right\} . $$
We focus on the situation where $e(\ell s)$ and $e((\ell + 1)s)$ are both equal to $\ell - s + \rr$, which is the critical case for our argument.

\begin{lemma}
\label{Lem:Counting}
Suppose $$e(\ell s) = e( \ell + 1)s = \ell - s + \rr,$$ for some $0 \leq \ell \leq r$.
If $\psi_{ij}$ is $\delta$-permissible on $\Gamma_{[a(\ell),b(\ell)]}$, then either
\begin{enumerate}
\item  $i < \ell < j$, and $i+j = \ell + \rr$, or
\item  $i = j = \ell$.
\end{enumerate}
\end{lemma}

\begin{proof}
Note that, by our choice of $D$,
\[
\deg(D_i|_{\Gamma_{[0,k]}}) = \left\{ \begin{array}{ll} i + k & \mbox{ for } i > \ell, \\
														i + \ell s & \mbox{ for } i = \ell, \\
														i + k - s & \mbox{ for } i < \ell. \end{array}  \right.
\]
Also, since $e(k)$ is nondecreasing,
$$ e(k) = \ell - s + \rr $$
for all $k$ in $[\ell s,(\ell +1)s]$, and in particular for $k$ in  $[a(\ell), b(\ell)]$.

We now prove the lemma in the case where $\ell \leq \rr$.  The proof in the case where $\ell > \rr$ is similar.  Suppose $i \geq \ell,$ $j > \ell,$ and $k \in [a(\ell), b(\ell)]$.  Then
\begin{align*}
 \deg (D_{ij} \vert_{\Gamma_{[0,k]}}) & \geq i + j + k + \ell s \\
                                                           &  > 2\ell + k + \ell s.
\end{align*}
On the other hand, we have
\begin{align*}
\deg (\Delta|_{\Gamma_{[0,k]}}) 	&= 2k + \ell - s + \rr \\
									&\leq 2\ell + k + \ell s,
									\end{align*}
where the inequality is given by using $k \leq b(\ell)$ and $b(\ell) =	\ell (s + 1) - \rr + s$.  Combining the two displayed inequalities shows that $\deg(D_{ij}|_{\Gamma_{[0,k]}})$ is greater than $\deg(\Delta|_{\Gamma_{[0,k]}})$, and hence $\psi_{ij}$ is not $\delta$-permissible on $\gamma_k^\circ$.

A similar argument shows that, if $i < \ell$, $j \leq \ell$, and $k \in [a(\ell), b(\ell)]$, then
\begin{align*}
\deg(D_{ij}|_{\Gamma_{[0,k-1]}})  	&\leq i + j + \ell s + k - 1 - s \\
									&< 2 \ell + \ell s + k - 1 - s.
\end{align*}
On the other hand, since $\ell \leq \rr$ by hypothesis, and $k \geq \ell s + 1$, we have
\begin{align*}
\deg(\Delta|_{\Gamma_{[0,k-1]}}) 	& = 2k -2 + \ell - s + \rr \\
									& \geq 2k -2 + 2\ell - s \\
									& \geq 2 \ell + \ell s + k - 1 - s.
\end{align*}
In this case, we conclude that $\deg(D_{ij}|_{\Gamma_{[0,k-1]}})$ is less than $\deg(\Delta|_{\Gamma_{[0,k-1]}})$, and hence $\psi_{ij}$ is not $\delta$-permissible on $\gamma_k^\circ$.

We have shown that, if $\psi_{ij}$ is $\delta$-permissible on $\Gamma_{[a(\ell),b(\ell)]}$, then either $i=j=\ell$ or $i < \ell < j$.  It remains to show that if $i < \ell < j$ then $i + j = \ell + \rr$.  Suppose $i < \ell < j$.  Then
\[
\deg(D_{ij}|_{\Gamma_{[0.k]}}) = i + j + 2k -s.
\]
If $\psi_{ij}$ is $\delta$-permissible on $\gamma_k^\circ$, then this is less than or equal to $\deg(\Delta|_{\Gamma_{[0,k]}})$, which is $2k + \ell - s + \rr$.  It follows that $i + j \leq \ell + \rr$.  Similarly, if $\psi_{ij}$ is $\delta$-permissible on $\gamma_k^\circ$ then $\deg(D_{ij}|_{\Gamma_{[0,k-1]}}) \geq \deg(\Delta|_{\Gamma_{[0,k-1]}})$, and it follows that $i + j \geq \ell + \rr$.  Therefore, $i + j = \ell + \rr$, as required.
\end{proof}

We continue with the notation from Lemma~\ref{Lem:Counting}, with $\ell$ a fixed integer between $0$ and $r$, and $[a(\ell), b(\ell)]$ the corresponding subinterval of $[\ell s + 1, (\ell + 1)s]$.  We also fix a subset $\cA \subset \{ (i,j) \in \ZZ^2 \ | \ 0 \leq i \leq j \leq r \}$ and suppose that $\theta(v) = \min \{ \psi_{ij}(v) \ | \ (i,j) \in \cA \}$ occurs at least twice at every point. Equivalently, in the set up of Lemma~\ref{Lem:Counting}, we assume that $b_{ij} \gg 0$ for $(i,j)$ not in $\cA$.

\begin{remark}  \label{rem:zeroloops}
The following proposition is the key technical step in our inductive argument, and may be seen as a generalization of the following two simple facts.  In order for the minimum to be achieved everywhere at least twice, on a chain of zero loops (i.e.\ a single edge), at least two functions are required, and on a chain of one loop, at least three functions are required (Lemma~\ref{Lem:FunctionsDisagree}).
\end{remark}

\begin{proposition}
\label{Prop:ConsecutiveBridges}
Suppose $ e(a(\ell)) = e(b(\ell)) = \ell  - s + \rr. $
Then there are at least $b(\ell)-a(\ell)+3$ functions $\psi_{ij}$, with $(i,j) \in \cA$, that are $\delta$-permissible on $\Gamma_{[a(\ell),b(\ell)]}$.  \end{proposition}

\begin{proof}
Let $a = a(\ell)$ and $b = b(\ell)$.  Assume that there are at most $b - a + 2$ functions that are $\delta$-permissible on $\Gamma_{[a, b]}$.  We will show that the bottom edge lengths $m_k$ for $k \in [a, b]$ satisfy a linear relation with small integer coefficients, contradicting the admissibility of the edge lengths of $\Gamma$ (Definition~\ref{Def:Admissible}).

Since $e(k)$ is nondecreasing, the assumption that $e(a) = e(b)$ implies that $\Delta$ contains exactly two points on each loop in $\Gamma_{[a,b]}$, and no points in the interiors of the bridges. It follows that $\theta$ has constant slope on each of these bridges.  As discussed in Section~\ref{Sec:Permissible}, the slope at the midpoint of $\beta_k$ is determined by the degree of $\ddiv \theta$ on $\Gamma_{[0,k]}$, and one computes that this slope is $2r - e(k)$.  Therefore, the slope of $\theta$ is constant on every bridge in $\Gamma_{[a,b]}$, and equal to $$\sigma := 2r - \ell + s - \rr.$$

We begin by describing the shapes of the $\delta$-permissible functions $\psi_{ij}$ on $\Gamma_{[a, b]}$.  By Lemma \ref{Lem:Counting}, the $\delta$-permissible functions $\psi_{ij}$ satisfy either $i=j=\ell$ or $i < \ell < j$ and $i + j = \ell + \rr$.  Suppose $i < \ell < j$.  In this case, the shape of $\psi_{ij}$ on the subgraph $\gamma_k$ is as pictured in Figure \ref{Fig:PermissibleShape}.

\begin{figure}[h!]
\begin{tikzpicture}
\matrix[column sep=0.5cm] {

\begin{scope}[grow=right, baseline]
\draw (-1,0) circle (1);
\draw (-3,0)--(-2,0);
\draw (0,0)--(1,0);
\draw [ball color=white] (-.5,.866) circle (0.55mm);
\draw (-.6, .6) node {$\tiny q_k$};
\draw [ball color=black] (-1.707,0.707) circle (0.55mm);
\draw [ball color=black] (-0.134,0.5) circle (0.55mm);
\draw (-2.5,-0.25) node {$\tiny{\sigma}$};
\draw (0.5,-0.25) node {$\tiny{\sigma}$};
\draw (-1,-1.25) node {$\tiny{\sigma-1}$};
\draw (-.15,.9) node {$\tiny{2}$};
\draw (0.15,0.3) node {$\tiny{1}$};
\draw (-2.15,0.3) node {$\tiny{1}$};
\draw (-1.2,1.22) node {$\tiny{0}$};
\end{scope}

\\};
\end{tikzpicture}
\caption{The shape of $\psi_{ij}$ on $\gamma_k$, for $i < \ell < j$.}
\label{Fig:PermissibleShape}
\end{figure}
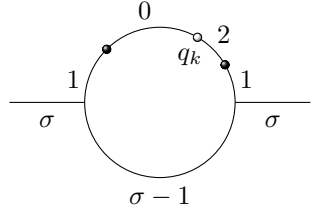

Note that the shape of $\psi_{ij}$ is determined by the shapes of $\psi_i$ and $\psi_j$, as shown in Figures~\ref{Fig:LeftHalf} and \ref{Fig:RightHalf}, respectively.  The point $q_k$ of $D$ on $\gamma_k^\circ$ is marked with a white circle.  The fact that the slopes of $\psi_{ij}$ along the bridges are equal to $\sigma$ is due to the condition $i + j = \ell + \rr$.

We now describe the shape of the function $\psi_{\ell \ell}$.  Note that the slope of $\psi_{\ell \ell}$ along the bridge $\beta_{\ell s}$ is $2r - 2 \ell$, and the slope increases by two along each successive bridge $\beta_k$, for $k \in [\ell s + 1, (\ell + 1) s]$.  It follows that if $\sigma$ is odd then $\psi_{\ell \ell}$ is $\delta$-permissible on only one loop $\gamma_h^\circ$, as shown in Figure~\ref{Fig:PermissibleShape2}.

\begin{figure}[h!]
\begin{tikzpicture}
\matrix[column sep=0.5cm] {

\begin{scope}[grow=right, baseline]
\draw (-1,0) circle (1);
\draw (-3,0)--(-2,0);
\draw (0,0)--(1,0);
\draw [ball color=white] (-.5,.866) circle (0.55mm);
\draw (-2.6,-0.25) node {$\tiny{\sigma-1}$};
\draw (0.6,-0.25) node {$\tiny{\sigma+1}$};
\draw (-1,-1.25) node {$\tiny{\sigma-1}$};
\draw (0.1,.55) node {$\tiny{2}$};
\draw (-1.65,1.05) node {$\tiny{0}$};
\end{scope}

\\};
\end{tikzpicture}
\caption{The shape of $\psi_{\ell \ell}$ on $\gamma_h$, when $\sigma$ is odd.}
\label{Fig:PermissibleShape2}
\end{figure}
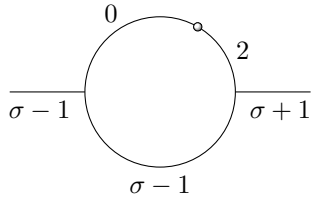

\noindent If $\sigma$ is even, then $\psi_{\ell \ell}$ is $\delta$-permissible on two consecutive loops $\gamma_h$ and $\gamma_{h+1}$, as shown in Figure \ref{Fig:PermissibleShape3}.
\begin{figure}[h!]
\begin{tikzpicture}
\matrix[column sep=0.5cm] {

\begin{scope}[grow=right, baseline]
\draw (-1,0) circle (1);
\draw (-3,0)--(-2,0);
\draw (0,0)--(1,0);
\draw [ball color=white] (-.5,.866) circle (0.55mm);
\draw (-2.6,-0.25) node {$\tiny{\sigma-2}$};
\draw (0.5,-0.25) node {$\tiny{\sigma}$};
\draw (-1,-1.25) node {$\tiny{\sigma-2}$};
\draw (.1,.55) node {$\tiny{2}$};
\draw (-1.65,1.05) node {$\tiny{0}$};
\draw (2,0) circle (1);
\draw (3,0)--(4,0);
\draw [ball color=white] (2.5,.866) circle (0.55mm);
\draw (3.6,-0.25) node {$\tiny{\sigma+2}$};
\draw (2,-1.25) node {$\tiny{\sigma}$};
\draw (3.1,.55) node {$\tiny{2}$};
\draw (1.35,1.05) node {$\tiny{0}$};
\end{scope}

\\};
\end{tikzpicture}
\caption{Ths shape of $\psi_{\ell \ell}$ on $\Gamma_{[h,h+1]}$, when $\sigma$ is even.}
\label{Fig:PermissibleShape3}
\end{figure}
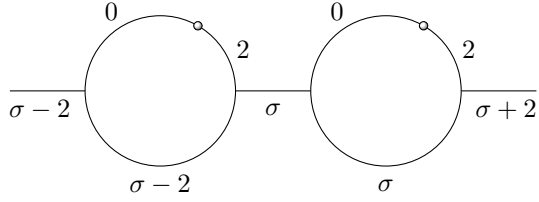
We choose $h$ so that $\gamma_h^\circ$ is the leftmost loop on which $\psi_{\ell \ell}$ is $\delta$-permissible.  We will use $v_h$ as a point of reference for the remaining calculations in the proof of the proposition.  (The values of $\psi_{ij}$ and $\psi_{\ell \ell}$ at every point in $\Gamma_{[a,b]}$ are determined by the shape computations above and the values at $v_h$.)

For the permissible functions $\psi_{ij}$ with $i < \ell < j$, the slopes along the bridges and bottom edges are independent of $(i,j)$.  One then computes directly that
\begin{equation} \label{first}
\psi_{ij}(q_k) - \psi_{i'j'}(q_k) = \psi_{ij}(v_h) - \psi_{i'j'}(v_h) + (i' - i) m_k.
\end{equation}
Similarly, one computes
\begin{equation} \label{second}
\psi_{ij}(q_h) - \psi_{\ell \ell}(q_h) = \psi_{ij}(v_h) - \psi_{\ell \ell}(v_h) + (r+ s - i - 1 - \sigma_h(\ell)) \cdot m_h,
\end{equation}
and, when $\sigma$ is even,
\begin{equation}  \label{third}
\psi_{ij}(q_{h+1}) - \psi_{\ell \ell}(q_{h+1}) = \psi_{ij}(v_h) - \psi_{\ell \ell}(v_h) + (r + s - i - 1 - \sigma_{h+1}(\ell)) \cdot m_{h+1} + m_h.
\end{equation}

We use these expressions, together with the tropical dependence hypothesis (our standing assumption that $\min \{ \psi_{ij}(v) \ | \ (i,j) \in \cA \}$ occurs at least twice at every point) to produce a linear relation with small integer coefficients among the bottom edge lengths $m_{a}, \ldots, m_{b}$, as follows.

\medskip

Let $\cA' \subset \cA$ be the set of pairs $(i,j)$ such that $\psi_{ij}$ is $\delta$-permissible on $\Gamma_{[a, b]}$.  We now build a graph whose vertices are the pairs $(i,j) \in \cA'$, and whose edges record the pairs that achieve the minimum at $v_h$ and the points $q_k$.  Say $\psi_{i_0j_0}$ and $\psi_{i'_0j'_0}$ achieve the minimum at $v_h$.  Then we add an edge $e_0$ from $(i_0,j_0)$ to $(i_0',j_0')$ in the graph.  Associated to this edge, we have the equation
\begin{equation}
\psi_{i_0j_0}(v_h) - \psi_{i_0'j_0'}(v_h) = 0. \tag{$E_0$}
\end{equation}

Next, for $a \leq k \leq b$, say $\psi_{i_k j_k}$ and $\psi_{i_k 'j_k'}$ achieve the minimum at $q_k$.  Then we add an edge $e_k$  from $(i_k, j_k)$ to $(i_k', j_k')$ and, associated to this edge, we have the equation
\begin{equation}
\psi_{i_k j_k}(v_h) - \psi_{i_k' j_k'}(v_h) = \alpha_k m_k + \lambda_k m_{k-1}, \tag{$E_k$}
\end{equation}
where $\alpha_k$ and $\lambda_k$ are small positive integers determined by the formula (\ref{first}), (\ref{second}), or (\ref{third}), according to whether one of the pairs is equal to $(\ell, \ell)$ and, if so, whether $k$ is equal to $h$ or $h+1$.  Note that, in every case, $\alpha_k$ is nonzero.

The graph now has $b - a + 2$ edges and, by hypothesis, it has at most $b - a + 2$ vertices.  Therefore, it must contain a loop.  If the edges $e_{k_1}, \ldots e_{k_t}$ form a loop then we can take a linear combination of the equations $E_{k_1}, ..., E_{k_t}$, each with coefficient $\pm 1$, so that the left hand sides add up to zero.  This gives a linear relation among the bottom edge lengths $m_{k_1}, \ldots, m_{k_t}$, with small integer coefficients.  Furthermore, if $k_t > k_j$ for all $j \neq t$, then $m_{k_t}$ appears with nonzero coefficient in $E_{k_t}$, and does not appear in $E_{k_j}$ for $j < t$, so this linear relation is nontrivial.  Finally, note that $\vert \alpha_k \vert \leq r+s \leq g$ for all $k$, and $\lambda_k$ is either 0 or 1, so the coefficient of each edge length $m_k$ is an integer of absolute value less than or equal to $g+1$.  This contradicts the hypothesis that $\Gamma$ has admissible edge lengths, and proves the proposition.
\end{proof}

\section{Proof of Theorem~\ref{thm:mainthm} for $\rho(g,r,d) = 0$}  \label{sec:rhozero}

In this section, we continue with the assumption from \S\ref{Sec:Excess} that $\rho(g,r,d) = 0$ and prove Conjecture~\ref{Conj:TropicalMRC} for $m = 2$, applying an inductive argument that relies on Lemma~\ref{Lem:Counting} and Proposition~\ref{Prop:ConsecutiveBridges} in the inductive step.  The case $\rho(g,r,d) > 0$ is handled by a minor variation on these arguments in \S\ref{sec:rhopos}.

\begin{remark}
Wang has recently shown that the maximal rank conjecture for $m = 2$ follows from the special case where $\rho(g,r,d) = 0$ \cite{Wang15}.  Our proof of Theorem~\ref{thm:mainthm} does not rely on this reduction.  We prove Conjecture~\ref{Conj:TropicalMRC} for $m = 2$ and arbitrary $\rho(g,r,d)$.
\end{remark}

We separate the argument into two cases, according to whether or not $\mu_2$ is injective.  The following identity is used to characterize the range of cases in which $\mu_2$ is injective and to count the set $\cA$ that we define in the remaining cases.

\begin{lemma} \label{lem:identity}
Suppose $s \leq r$.  Then
\[
{r+2 \choose 2} - {r-s \choose 2} + {s \choose 2} = 2d - g +1.
\]
\end{lemma}

\begin{proof}
The lemma follows from a series of algebraic manipulations. Expand the left hand side as a polynomial in $r$ and $s$, collect terms, and apply the identities $s = g-d+r$ and $g = (r+1)s$.
\end{proof}

It follows immediately from Lemma~\ref{lem:identity} that
\[
{r+2 \choose 2} \leq 2d-g+1 \mbox{ \ \ if and only if \ \ } r -s \leq s.
\]
In particular, the maximal rank conjecture predicts that $\mu_2$ is injective for a general linear series on a general curve exactly when $r \leq 2s$.   We now proceed with the proof that $\{ \psi_{ij} \ | \ 0 \leq i \leq j \leq r\}$ is tropically independent in the injective case.

\bigskip

\begin{proof}[Proof of Conjecture~\ref{Conj:TropicalMRC} for $m = 2$, $\rho(g,r,d) = 0$, and $r \leq 2s$]

We must show that the set of functions $\{ \psi_{ij} \ | \ 0 \leq i \leq j \leq r \}$ is tropically independent.  Suppose not.  Then there are constants $b_{ij}$ such that the minimum
\[
\theta (v) = \min_{ij} \psi_{ij}(v) + b_{ij}
\]
occurs at least twice at every point $v \in \Gamma$.  We continue with the notation from \S\ref{Sec:Excess}, setting
\[
\Delta = 2D + \ddiv \theta, \ \ \delta_i = \deg(\Delta)|_{\gamma_i}, \ \mbox { and } \ \ e(k) = \delta_0 + \cdots + \delta_k - 2k.
\]
As described above, our strategy is to bound the excess degree function $e(k) = \delta_0 + \cdots + \delta_k - 2k$ inductively, moving from left to right across the graph.

More precisely, we claim that
\begin{equation}
e(\ell s) \geq \ell -s + \rr \ \mbox{ for } \ \ \ell \leq r.  \label{claim1}
\end{equation}
 We prove this claim by induction on $\ell$, using Lemma~\ref{Lem:Counting} and Proposition~\ref{Prop:ConsecutiveBridges}.  To see that the theorem follows from the claim, note that the claim implies that
\[
\deg(\Delta) \geq 2g + r-s + \rr + 2.
\]
Since $d = g + r -s$, this gives $\deg(\Delta) \geq 2d + s - \lfloor \frac{r}{2} \rfloor + 2$, a contradiction, since $r \leq 2s$.  It remains to prove the claim (\ref{claim1}).

The claim is clear for $\ell = 1$, since $e(k) \geq \delta_0 \geq 2$ for all $k$ and $\rr \leq s$, by assumption.  We proceed by induction on $\ell$.  Assume that $\ell < 2r - s - 1 - \rr$ and
$$ e(\ell s) \geq \ell -s + \rr. $$
We must show that $e((\ell + 1) s) \geq \ell -s + \rr + 1$.  If $e(\ell s) > \ell - s + \rr$ then there is nothing to prove, since $e$ is nondecreasing.  It remains to rule out the possibility that $e(\ell s) = e((\ell + 1) s) = \ell - s + \rr$.

Suppose that $e(\ell s) = e((\ell + 1) s) = \ell - s + \rr$.  Fix $a = a(\ell)$ and $b = b(\ell)$ as in \S\ref{Sec:Excess}.  By Lemma~\ref{Lem:Counting}, if $\psi_{ij}$ is $\delta$-admissible on $\Gamma_{[a, b]}$ then either $i = j = \ell$ or $i < \ell < j$ and $i + j = \ell + \rr$.  We consider two cases and use Proposition~\ref{Prop:ConsecutiveBridges} to reach a contradiction in each case.

\medskip

\noindent \emph{Case 1:} If $1 \leq \ell \leq \rr$ then there are exactly $\ell + 1$ possibilities for $i$, and $j$ is uniquely determined by $i$.  In this case $b - a  = \ell + s - \rr -1$.  Since $r \leq 2s$, this implies that the number of $\delta$-permissible functions is at most $b - a + 2$, which contradicts Proposition~\ref{Prop:ConsecutiveBridges}, and the claim follows.

\medskip

\noindent \emph{Case 2:} If $\rr < \ell < r$ then there are exactly $r - \ell + 1$ possibilities for $j$, and $i$ is uniquely determined by $j$.  In this case, $b - a = s - \ell + \rr + 1$, which is at least $r - \ell - 1$, since $r \leq 2s$.  Therefore, the number of $\delta$-permissible functions on $\Gamma_{[a, b]}$ is at most $b - a + 2$, which contradicts Proposition~\ref{Prop:ConsecutiveBridges}, and the claim follows.
\medskip

\noindent This completes the proof of Conjecture~\ref{Conj:TropicalMRC} (and hence Theorem~\ref{thm:mainthm}) in the case where $m = 2$, $\rho(g,r,d) = 0$, and $r \leq 2s$.
\end{proof}

Our proof of Conjecture~\ref{Conj:TropicalMRC} for $m = 2$, $\rho(g,r,d) = 0$, and $r > 2s$ is similar to the argument above, bounding the excess degree function $e(\ell s)$ by induction on $\ell$, with Lemma~\ref{Lem:Counting} and Proposition~\ref{Prop:ConsecutiveBridges} playing a key role in the inductive step.  The one essential new feature is that we must specify the subset $\cA$.  The description of this set, and the argument that follows, depend in a minor way on the parity of $r$, so we fix
\[
\epsilon(r) = \left\{ \begin{array}{ll} 0 & \mbox{if $r$ is even,}\\
						    1 & \mbox{if $r$ is odd.} \end{array} \right.
\]

\medskip

Let $\cA$ be the subset of the integer points in the triangle $0 \leq i \leq j \leq r$ that are not in any of the following three regions:
\begin{enumerate}
\item the half-open triangle where $j \geq i + 2$ and $i + j < r-2s + \epsilon(r)$,
\item the half-open triangle where $j \geq i + 2$ and $i + j > r + 2s$,
\item the closed chevron where $r-s + \epsilon(r) \ \leq \ i + j \leq r+s$, and either $$i \leq \frac{1}{2}(r -2s-2+ \epsilon(r)) \mbox{ \ \ or \ \ }j \geq \frac{1}{2}(r+2s+2).$$
\end{enumerate}
Figure~\ref{Fig:A} illustrates the case $g = 36$, $r=11$, $d = 44$, and $s = 3$.  The points of $\cA$ are marked with black dots, the three regions are shaded gray, and the omitted integer points are marked with white circles.

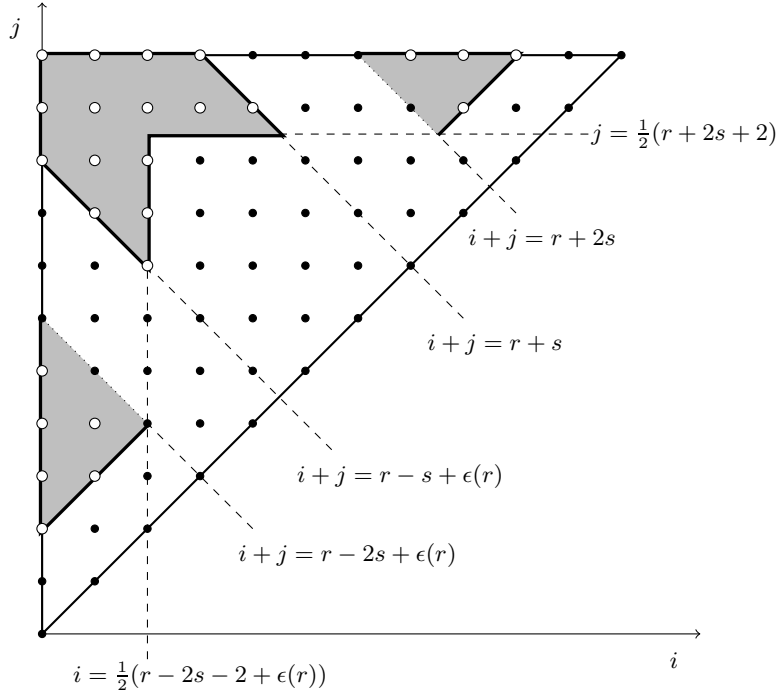
\begin{figure}[h!]
\begin{tikzpicture}[scale=0.7]


\draw[->] (0,0)--(12.5,0);
\draw[->] (0,0)--(0,12);
\draw [thick] (0,0)--(11,11)--(0,11)--(0,0);

\draw [line width=2.5] (0,6)--(0,2)--(2,4);

\draw [line width=2.5] (7.5,9.5)--(9,11)--(6,11);

\draw [line width=2.5] (0,11)--(0,9)--(2,7)--(2,9.5)--(4.5,9.5)--(3,11)--(0,11);

\fill[lightgray] (0,2) -- (2,4) -- (0,6) -- cycle;
\fill[lightgray] (0,9)--(2,7)--(2,9.5)--(4.5,9.5)--(3,11) --(0,11)--cycle;
\fill[lightgray]  (6,11)--(7.5,9.5)--(9,11)--cycle;

\draw [fill=white] (0,2) circle [radius=0.1];
\draw [fill=white] (0,3) circle [radius=0.1];
\draw [fill=white] (0,4) circle [radius=0.1];
\draw [fill=white] (0,5) circle [radius=0.1];	
\draw [fill=white] (1,3) circle [radius=0.1];	
\draw [fill=white] (1,4) circle [radius=0.1];

\draw [fill=white] (0,9) circle [radius=0.1];	
\draw [fill=white] (0,10) circle [radius=0.1];	
\draw [fill=white] (0,11) circle [radius=0.1];
	
\draw [fill=white] (1,8) circle [radius=0.1];			
\draw [fill=white] (1,9) circle [radius=0.1];	
\draw [fill=white] (1,10) circle [radius=0.1];	
\draw [fill=white] (1,11) circle [radius=0.1];	

\draw [fill=white] (2,7) circle [radius=0.1];	
\draw [fill=white] (2,8) circle [radius=0.1];	
\draw [fill=white] (2,9) circle [radius=0.1];	
\draw [fill=white] (2,10) circle [radius=0.1];	
\draw [fill=white] (2,11) circle [radius=0.1];	

\draw [fill=white] (3,10) circle [radius=0.1];	
\draw [fill=white] (3,11) circle [radius=0.1];
	
\draw [fill=white] (4,10) circle [radius=0.1];	

\draw [fill=white] (7,11) circle [radius=0.1];	
\draw [fill=white] (8,10) circle [radius=0.1];	
\draw [fill=white] (8,11) circle [radius=0.1];	
\draw [fill=white] (9,11) circle [radius=0.1];

\draw [fill=black] (0,0) circle [radius=0.07];	
\draw [fill=black] (0,1) circle [radius=0.07];
\draw [fill=black] (0,6) circle [radius=0.07];	
\draw [fill=black] (0,7) circle [radius=0.07];	
\draw [fill=black] (0,8) circle [radius=0.07];	
	
\draw [fill=black] (1,1) circle [radius=0.07];	
\draw [fill=black] (1,2) circle [radius=0.07];	
\draw [fill=black] (1,5) circle [radius=0.07];	
\draw [fill=black] (1,6) circle [radius=0.07];	
\draw [fill=black] (1,7) circle [radius=0.07];	

\draw [fill=black] (2,2) circle [radius=0.07];		
\draw [fill=black] (2,3) circle [radius=0.07];	
\draw [fill=black] (2,4) circle [radius=0.07];	
\draw [fill=black] (2,5) circle [radius=0.07];	
\draw [fill=black] (2,6) circle [radius=0.07];	

\draw [fill=black] (3,3) circle [radius=0.07];	
\draw [fill=black] (3,4) circle [radius=0.07];	
\draw [fill=black] (3,5) circle [radius=0.07];	
\draw [fill=black] (3,6) circle [radius=0.07];	
\draw [fill=black] (3,7) circle [radius=0.07];	
\draw [fill=black] (3,8) circle [radius=0.07];	
\draw [fill=black] (3,9) circle [radius=0.07];	

\draw [fill=black] (4,4) circle [radius=0.07];
\draw [fill=black] (4,5) circle [radius=0.07];
\draw [fill=black] (4,6) circle [radius=0.07];
\draw [fill=black] (4,7) circle [radius=0.07];
\draw [fill=black] (4,8) circle [radius=0.07];
\draw [fill=black] (4,9) circle [radius=0.07];
\draw [fill=black] (4,11) circle [radius=0.07];

\draw [fill=black] (5,5) circle [radius=0.07];
\draw [fill=black] (5,6) circle [radius=0.07];
\draw [fill=black] (5,7) circle [radius=0.07];
\draw [fill=black] (5,8) circle [radius=0.07];
\draw [fill=black] (5,9) circle [radius=0.07];
\draw [fill=black] (5,10) circle [radius=0.07];
\draw [fill=black] (5,11) circle [radius=0.07];

\draw [fill=black] (6,6) circle [radius=0.07];
\draw [fill=black] (6,7) circle [radius=0.07];
\draw [fill=black] (6,8) circle [radius=0.07];
\draw [fill=black] (6,9) circle [radius=0.07];
\draw [fill=black] (6,10) circle [radius=0.07];
\draw [fill=black] (6,11) circle [radius=0.07];

\draw [fill=black] (7,7) circle [radius=0.07];
\draw [fill=black] (7,8) circle [radius=0.07];
\draw [fill=black] (7,9) circle [radius=0.07];
\draw [fill=black] (7,10) circle [radius=0.07];

\draw [fill=black] (8,8) circle [radius=0.07];
\draw [fill=black] (8,9) circle [radius=0.07];

\draw [fill=black] (9,9) circle [radius=0.07];
\draw [fill=black] (9,10) circle [radius=0.07];

\draw [fill=black] (10,10) circle [radius=0.07];
\draw [fill=black] (10,11) circle [radius=0.07];

\draw [fill=black] (11,11) circle [radius=0.07];

\node at (5.8,1.5) {\small $i+j= r-2s +\epsilon(r)$};
\draw[dashed] (4,2)--(2,4);
\draw[dotted] (2,4)--(0,6);

\node at (6.8,3) {\small $i+j= r-s +\epsilon(r)$};
\draw[dashed] (5.5,3.5)--(2.1,6.9);

\node at (8.6,5.5) {\small $i+j=r+s$};
\draw[dashed] (8,6)--(4.5,9.5);

\node at (9.5,7.5) {\small $i+j=r+2s$};
\draw[dashed] (9,8)--(7.5,9.5);
\draw[dotted] (7.5,9.5)--(6,11);

\node at (12,-0.5) {\small $i$};
\node at (-0.5,11.5) {\small $j$};

\node at (3,-0.8) {\small $i = \frac{1}{2}(r -2s-2+ \epsilon(r))$};
\draw[dashed] (2,6.9)--(2,-0.5);

\node at (12.2,9.5) {\small $j = \frac{1}{2}(r+2s+2)$};
\draw[dashed] (4.5,9.5)--(10.5,9.5);

\end{tikzpicture}
\caption{Points in the set $\cA$ are marked by black dots.  The integer points in the triangle $0 \leq i \leq j \leq r$ that are omitted from $\cA$ are marked with white circles.}
\label{Fig:A}
\end{figure}

\begin{remark}
There are many possible choices for $\cA$, as one can see even in relatively simple cases, such as Example~\ref{Ex:QuadricRelation}.  We present one particular choice that works uniformly for all $g$, $r$, and $d$.  (In the situation of Example~\ref{Ex:QuadricRelation}, the two half-open triangles are empty, and the closed chevron contains a single integer point, namely $(0,3)$.)  The essential property for the purposes of our inductive argument is the number of points $(i,j)$ in $\cA$, with $i \neq j$, on each diagonal line $i + j = k$, for $0 \leq k \leq 2r$.  The argument presented here works essentially verbatim for any other subset of the integer points in the triangle with this property, and can be adapted to work somewhat more generally.  We have made no effort to characterize those subsets that are tropically independent, since producing a single such subset is sufficient for the proof of Theorem~\ref{thm:mainthm}.
\end{remark}

\begin{remark}
Our choice of $\cA$, suitably intepreted, works even in the injective case.  When $r - s \leq s$, the shaded regions are empty, since the half space $i \leq \frac{1}{2}(r -2s-2+ \epsilon(r))$ lies entirely to the left of the triangle $0 \leq i \leq j \leq r$, and the half space $j \geq \frac{1}{2}(r+2s+2)$ lies above it.
\end{remark}

We now verify that the set $\cA$ described above has the correct size.

\begin{lemma}
The size of $\cA$ is
\[
\# \cA = 2d-g+1.
\]
\end{lemma}

\begin{proof}
As shown in Figure~\ref{Fig:A}, moving the lower left triangle vertically and the upper right triangle horizontally by integer translations, we can assemble the shaded regions to form a closed triangle minus a half-open triangle.  These translations show that the two half-open triangles plus the convex hull of the chevron shape are scissors congruent to a triangle of side length $r -s -2$ that contains ${r-s \choose 2}$ integer points.  The difference between the chevron shape and its convex hull is a half-open triangle that contains ${s \choose 2}$ integer points.  Therefore, the shaded region contains exactly ${r-s \choose 2} - {s \choose 2}$ lattice points, and the proposition then follows from the identity in Lemma~\ref{lem:identity}.
\end{proof}

\begin{proof}[Proof of Conjecture~\ref{Conj:TropicalMRC} for $m = 2$, $\rho(g,r,d) = 0$, and $r > 2s$.]
We will show that $$\{ \psi_{ij} \ | \ (i,j) \in \cA \}$$ is tropically independent.  Suppose not.  Then there are constants $b_{ij}$ such that $\theta(v) = \min_{(i,j) \in \mathcal{A}} \{ \psi_{ij}(v) + b_{ij} \}$ occurs at least twice at every point $v$ in $\Gamma$.  Let
$$ \Delta = \ddiv(\theta) + 2D, \mbox{ \ \ } \delta_i = \deg(\Delta)|_{\gamma_i}, \mbox{ \ and \ } e(k) = \delta_0 + \cdots + \delta_k - 2k. $$

Note that $\deg_{w_0}(\Delta)$ is $2r - \sigma_0(\theta)$, where $\sigma_0(\theta)$ is the outgoing slope of $\theta$ at $w_0$.  Since the minimum is achieved twice at every point, this slope must agree with the slope $\sigma_0(\psi_{ij}) = 2r - i - j$ for at least two pairs $(i,j) \in \cA$.   The points in the half-open triangle where $j \geq i + 2$ and $i + j > r + 2s$ are omitted from $\cA$, so there is only one pair $(i,j) \in \cA$ such that $i + j = k$, for $k < r - 2s + \epsilon(r)$.  It follows that $\deg_{w_0}(\Delta) \geq r-2s + \epsilon(r)$.  Similarly, $\deg_{v_{g+1}}(\Delta) \geq r - 2s$.

We claim that
\begin{equation}
e(\ell s) \geq \ell - s + \rr \ \mbox{ for } \ \ell \leq \left \lfloor \frac{r}{2} \right \rfloor + s + 1.  \label{claim2}
\end{equation}
Note that the assumption $r > 2s$ implies that $ \left \lfloor \frac{r}{2} \right \rfloor + s + 1 \leq r$.  Since $e$ is a nondecreasing function of $k$, and $\deg_{v_{g+1}}(\Delta) \geq r - 2s$, the claim implies that
\[
\deg(\Delta) \geq 2g + \left (  \left \lfloor \frac{r}{2} \right \rfloor + s + 1 - s + \rr \right) + r - 2s .
\]
Collecting terms gives $\deg(\Delta) \geq 2g + 2r - 2s +1= 2d + 1$, a contradiction.

It remains to prove claim (\ref{claim2}).  Since $\delta_0 \geq r-2s+\epsilon (r)$, the claim holds for $\ell \leq \rr - s$.  We proceed by induction on $\ell$.  Assume that
$ e(\ell s) \geq \ell -s + \rr$ and $\ell \leq \left \lfloor \frac{r}{2} \right \rfloor + s$.
We must show that $e((\ell + 1) s \geq \ell -s + \rr + 1$.  If $e(\ell s) > \ell - s + \rr$ then there is nothing to prove, since $e$ is nondecreasing.  It remains to rule out the possibility that $e(\ell s) = e((\ell + 1) s) = \ell - s + \rr$.

Suppose $e(\ell s) = e((\ell + 1) s) = \ell - s + \rr$.  Fix $a=a(\ell)$ and $b=b(\ell)$ as in \S\ref{Sec:Excess}.  By Lemma~\ref{Lem:Counting}, if $\psi_{ij}$ is $\delta$-admissible on $\Gamma_{[a,b]}$ then either $i = j = \ell$ or $i < \ell < j$ and $i + j = \ell + \rr$.  We consider three cases.

\medskip

\noindent \emph{Case 1:}  If $\rr - s \leq \ell \leq \left \lfloor \frac{r}{2} \right \rfloor $ then there are $\rr - s$ pairs $(i,j)$ with $i \neq j$ and $i+j = \ell + \rr$ that are contained in the closed chevron and hence omitted from $\cA$.  This leaves $$ \ell + 1 - \rr + s = b-a + 2 $$ pairs $(i,j) \in \cA$ such that $\psi_{ij}$ is $\delta$-permissible on $\Gamma_{[a,b]}$.  We can then apply Proposition~\ref{Prop:ConsecutiveBridges}, and the claim follows.

\medskip

\noindent \emph{Case 2:}   If $\left \lfloor \frac{r}{2} \right \rfloor <  \ell < \frac{r}{2} + s$ then there are $\left \lfloor \frac{r}{2} \right\rfloor - s$ pairs $(i,j)$ with $i+j = \ell + \rr$ that are in the closed chevron and hence omitted from $\cA$.  This leaves
$$ r - \ell + 1 - \left \lfloor \frac{r}{2} \right \rfloor + s = b- a + 2 $$
pairs $(i,j) \in \cA$ such that $\psi_{ij}$ is $\delta$-permissible on $\Gamma_{[a,b]}$. We can then apply Proposition~\ref{Prop:ConsecutiveBridges}, and the claim follows.

\medskip

\noindent \emph{Case 3:}  If $\ell = \frac{r}{2} + s$, then there are $\frac{r}{2} - s$ pairs $(i,j)$ with $i+j = r+s$ that are contained in the closed chevron and hence omitted from $\cA$.  This leaves one pair $(i,j) \in \cA$ such that $\psi_{ij}$ has slope $r$ on the bridge $\beta_{(\frac{r}{2} + s)s}$.  It follows that $\theta$ cannot have slope $r$ at any point of this bridge.  If $e((\ell +1)s) \leq r$, however, then the inductive hypothesis implies that $e(\ell s) = e(\ell s+1) = r$, hence $\theta$ has constant slope $r$ on this bridge, a contradiction, and the claim follows.
\end{proof}

\begin{remark}
In Case 3 of the above argument, the formulas for $a(\ell)$ and $b(\ell)$ would give $a(\ell) = b(\ell) + 1$, so the subgraph $\Gamma_{[a(\ell), b(\ell)]}$ might be thought of as a chain of $b(\ell) - a(\ell) + 1 = 0$ loops.  The inductive step in this case is then an application of a degenerate version of Proposition~\ref{Prop:ConsecutiveBridges} for a chain of zero loops, i.e.\ for a single edge.  See also Remark~\ref{rem:zeroloops}.
\end{remark}

\section{Proof of Theorem~\ref{thm:mainthm} for $\rho(g,r,d) > 0$}  \label{sec:rhopos}

Fix $\rho = \rho(g,r,d)$, $g' = g - \rho$, and $d' = d -\rho$.  Let $\Gamma'$ be a chain of $g'$ loops with admissible edge lengths.  Note that $\rho(g', r, d') = 0$. Therefore, the constructions in \S\ref{Sec:Excess}-\ref{sec:rhozero} produce a divisor $D'$ on $\Gamma'$ of rank $r$ and degree $d'$ whose class is vertex avoiding, together with a set $\cA'$ of integer points $(i,j)$ with $0 \leq i \leq j \leq r $ of size
\[
\# \cA' \ = \ \min \left \{ {r + 2 \choose 2}, 2d' - g' + 1 \right \} \ = \ \min \left \{ {r + 2 \choose 2}, 2d - g + 1 - \rho \right \},
\]
such that the collection of piecewise linear functions $\{ \psi'_{ij} \in R(D') \ | \ (i,j) \in \cA' \}$ is tropically independent.

We use $\Gamma'$, $D'$, and $\cA'$ as starting points to construct a chain of $g$ loops with admissible edge lengths $\Gamma$, a divisor $D$ of degree $d$ and rank $r$ whose class is vertex avoiding, and a set $\cA$ of size $\# \cA = \min \left \{ {r+2 \choose 2}, 2d - g + 1 \right \}$ such that $\{ \psi_{ij} \in R(D) \ | \ (i,j) \in \cA \}$ is tropically independent.  Note that
\[
g = g' + \rho, \ \ \ d = d' + \rho, \ \ \mbox{ and } \ \ \# \cA  - \# \cA' = \min \left \{ \rho, {r + 2 \choose 2} - \# \cA' \right \}.
\]

\begin{proof}[Proof of Conjecture~\ref{Conj:TropicalMRC} for $m = 2$, $\rho(g,r,d) > 0$, and ${r + 2 \choose 2} \geq 2d - g + 1$.]

We construct $\Gamma$, $D$ and $\cA$ by adding $\rho$ new loops to $\Gamma'$, $\rho$ new points to $D'$, and $\rho$ new points to $\cA'$.  Any collection of $\rho$ points in the complement of $\cA'$ will work, but the location of the new loops added to $\Gamma'$ depends on the set $\cA  \smallsetminus \cA'$.

Recall that the complement of $\cA'$ consists of the integer points in the closed chevron, the lower left half-open triangle, and the upper right half-open triangle, as shown in Figure~\ref{Fig:A}.  Suppose $\cA \smallsetminus \cA'$ consists of $\nu$ new points in the chevron, $\nu_1$ new points in the lower left half-open triangle, and $\nu_2$ new points in the upper right half-open triangle.  Then construct $\Gamma$ from $\Gamma'$ by adding $\nu_1$ new loops to the left end of $\Gamma'$, $\nu_2$ new loops to the right end of $\Gamma'$, and $\nu$ new loops in the middle of the chain, at locations that are specified as follows.

For $\rr - s \leq \ell <  \rr + s$, let $a(\ell)$ and $b(\ell)$ be as defined in \S\ref{Sec:Excess}.  For each new element $(i,j)$ from the chevron, we add a corresponding loop to the end of the subgraph $\Gamma'_{[a(\ell), b(\ell)]}$, where $\ell$ is the unique integer such that $i + j = \ell + \rr$.  In other words, if there are $t$ points $(i,j)$ in $\cA \smallsetminus \cA'$ such that $i + j = \ell + \rr$, we add $t$ new loops immediately to the right of the $b(\ell)$th loop in $\Gamma'$.

Let $\alpha(k)$ denote the number of new points $(i,j) \in \cA \smallsetminus \cA'$ such that $i + j \leq k$.  We construct our divisor $D'$ so that it has one chip on each of the new loops.  The new loops correspond to lingering steps in the associated lattice path, and the location of the points on the new loops are chosen in specific regions on the top edges, as described below, and sufficiently general so that the class of $D'$ is vertex avoiding.

Just as in \S\ref{Sec:Excess}-\ref{sec:rhozero}, we suppose that $\{ \psi_{ij} \ | \ (i,j) \in \cA\}$ is tropically dependent, choose constants $b_{ij}$ such that the minimum
\[
\theta(v) = \min_{ij} \{ \psi_{ij}(v) + b_{ij}) \}
\]
occurs at least twice at every point $v$ in $\Gamma$, and fix
$$ \Delta = \ddiv(\theta) + 2D, \mbox{ \ \ } \delta_i = \deg(\Delta)|_{\gamma_i}, \mbox{ \ \ } e(k) = \delta_0 + \cdots + \delta_k - 2k. $$  We again fix $s = g - d + r$, which is the same as $g' - d' + r$.  We claim that
\begin{enumerate}
\item $\delta_0 + \cdots + \delta_{\nu_1} \geq r - 2s + \epsilon(r)$,  \label{claim:triangle}
\item $e( \ell s + \alpha(\ell + \rr)) \geq \ell - s + \rr \mbox{ for } \left \lfloor \frac{r}{2} \right \rfloor -s + \epsilon(r) \leq \ell \leq \left \lfloor \frac{r}{2} \right \rfloor + s + 1$, \label{claim:chevron}
\item $\delta_{g - \nu_2 + 1} + \cdots + \delta_{g+1} \geq r - 2s.$  \label{claim:triangle2}
\end{enumerate}
Just as in the proof for $\rho = 0$ and $r > 2s$, the claim implies that $\deg(\Delta) \geq 2d + 1$, which is a contradiction.  It remains to prove the claim, which we do inductively, moving from left to right across the graph.

\medskip

To prove (\ref{claim:triangle}), we show that $e(\alpha(k)) \geq k$, for $0 \leq k \leq r - 2s + \epsilon(r)$.
For $k = 0$, there is nothing to prove, and we proceed by induction on $k$. Let $a = \alpha(k) + 1$ and $b = \alpha(k+1)$.  As in the $\rho = 0$ case, we must rule out the possibility that $e(a) = e(b) = k$.  Just as in Lemma~\ref{Lem:Counting}, if $e(a) = e(b) = k$ then the $\delta$-permissible functions on $\Gamma_{[a,b]}$ are exactly those $\psi_{ij}$ such that $i + j = k$.  We choose the location of the new points on the loops in $\Gamma_{[a,b]}$ so that the functions $\psi_i$ for $i \leq \frac{k}{2}$ have the combinatorial shape shown in Figure~\ref{Fig:LeftHalf}, on each loop in $\Gamma_{[a,b]}$, and those for $i > \frac{k}{2}$ have the combinatorial shape shown in Figure~\ref{Fig:RightHalf}.  It follows that each $\delta$-permissible $\psi_{ij}$ has the combinatorial shape shown in Figure~\ref{Fig:PermissibleShape}.  By construction, there are exactly $ b - a + 2$ pairs $(i,j) \in \cA$ such that $i + j = k$.  Then, just as in Proposition~\ref{Prop:ConsecutiveBridges}, we conclude that $e(b) \geq k + 1$, which proves (\ref{claim:triangle}).  (The argument in this case is somewhat simpler than in Proposition~\ref{Prop:ConsecutiveBridges}, since the combinatorial shapes appearing in Figures~\ref{Fig:PermissibleShape2} and \ref{Fig:PermissibleShape3} do not occur.)  The proof of (\ref{claim:triangle2}) is similar.

\medskip

It remains to prove (\ref{claim:chevron}).  Note that (\ref{claim:chevron}) follows from (\ref{claim:triangle}) for $\ell = \left \lfloor \frac{r}{2} \right \rfloor -s + \epsilon(r)$.  We proceed by induction on $\ell$.  Let $a = a(\ell) + \alpha(\ell + \rr - 1)$ and let $b = b(\ell) + \alpha(\ell + \rr)$. As in the $\rho = 0$ case, it suffices to rule out the possibility that $e(a) = e(b) = \ell - s + \rr$.

Suppose $e(a) = e(b) = \ell - s + \rr$.  Then, just as in Lemma \ref{Lem:Counting}, if $\psi_{ij}$ is $\delta$-permissible on $\Gamma_{[a,b]}$, then either $i = j = \ell$ or $i < \ell < j$ and $i + j = \ell + \rr$.  We choose the location of the points on the new loops in $\Gamma_{[a,b]}$ so that $\psi_{ij}$ has the shape shown in Figure~\ref{Fig:PermissibleShape} for $i < \ell < j$.  Then, just as in Proposition~\ref{Prop:ConsecutiveBridges}, it follows that there must be at least $b - a + 3$ functions that are $\delta$-permissible on $\Gamma_{[a,b]}$.  However, by construction, there are only $b -a + 2$ functions that are $\delta$-permissible on $\Gamma_{[a,b]}$, a contradiction.  We conclude that $e(b) > \ell - s + \rr$, as required.  This completes the proof of the claim, and the theorem follows.
\end{proof}

\begin{remark}
The analogue of (\ref{claim:triangle}) in the case $\rho(g,r,d) = 0$ and $r > 2s$ is the lower bound $\delta_0 \geq r - 2s + \epsilon(r)$ which comes from having only one pair $(i,j) \in \cA$ such that $\psi_{ij}$ has a given slope $\sigma$ at $w_0$, for $0 \leq \sigma < r - 2s + \epsilon(r)$.  This bound may be seen as coming from $r - 2s + \epsilon(r)$ applications of the degenerate version of Proposition~\ref{Prop:ConsecutiveBridges} for a chain of zero loops, i.e.\ a single edge.  As we add points to $\cA$ and add loops to the left of $w_0$, these chains of zero loops become actual chains of loops, and we then use the usual version of Proposition~\ref{Prop:ConsecutiveBridges}.  A similar remark applies to (\ref{claim:triangle2}).
\end{remark}

\begin{proof}[Proof of Conjecture~\ref{Conj:TropicalMRC} for $m = 2$, $\rho(g,r,d) > 0$, and ${r + 2 \choose 2} \leq 2d -g + 1$.]

Again, it suffices to construct a divisor $D$ on $\Gamma$ of rank $r$ and degree $d$ whose class is vertex avoiding such that all of the functions $\psi_{ij}$ are tropically independent.  Let
$$ \eta = \min \left \{ \rho , 2d-g+1-{{r+2}\choose{2}} \right \} . $$
By the arguments in the preceding case, on the chain of $g-\eta$ loops with bridges, there exists a vertex avoiding divisor $D'$ of rank $r$ and degree $d-\eta$ such that the functions $\psi_{ij}$ are tropically independent.  We construct a divisor $D$ on $\Gamma$ of rank $r$ and degree $d$ by specifying that $D\vert_{\Gamma_{[0,g-\eta]}} = D'$, and the remaining $\eta$ steps of the corresponding lattice path are all lingering, with the points on the last $\eta$ loops chosen sufficiently general so that the class of $D$ is vertex avoiding.  Then the restrictions of the functions $\psi_{ij}$ to $\Gamma_{[0,g-\eta]}$ are tropically independent, so the functions themselves are tropically independent as well.
\end{proof}

\bibliography{math}

\end{document}